\documentclass[12pt,reqno]{amsart}          
\usepackage{pgfplots}

\usepackage{booktabs}
\usepackage{xcolor}
\usepackage{color}

\numberwithin{equation}{section}

\usepackage[colorlinks = true,
            linkcolor = blue!60!black,
            urlcolor  = blue!40!black,
            citecolor = blue!60!black,
            anchorcolor = blue!60!black]{hyperref}

 \newcommand{\blue}[1]{\textcolor{black}{#1}}

\definecolor{light-gray}{gray}{0.95}
\definecolor{rodi}{rgb}{0,1,0}
\definecolor{orange}{rgb}{1,0.5,0}
\definecolor{LightCyan}{rgb}{0.88,1,1}
\definecolor{morado}{rgb}{0.44, 0.16, 0.39}
\newcommand\x{\times}

\newcommand{\dt}{\partial_t } 
\newcommand{\dx}{\partial_x }

\newcommand{\D}{\mathcal{D}}
\newcommand{\dd}{\partial} 

\newcommand{\veps}{\varepsilon}

\newcommand{\grad}{\operatorname{grad}}
\renewcommand{\div}{\operatorname{div}}
\newcommand{\bb}[1]{ \begin{bmatrix}#1\end{bmatrix}}

\newcommand{\RRR}{\mathbb{R}}

\newcommand{\wt}{\tilde{w}}

\newcommand{\dom}{\mathop{\mathrm{dom}}}
\newcommand{\V}{\mathcal{V}}

\newcommand{\lprp}[1]{\sideset{^\perp}{}{\mathop{#1}}}

\newcommand{\ip}[1]{\langle{#1}\rangle}

\newcommand{\om}{\Omega}

\usepackage{tikz}
\usetikzlibrary{decorations.text}
\usetikzlibrary{calc}
\usepackage[T1]{fontenc}
\usepackage[font=small,labelfont=bf,tableposition=top]{caption}

\DeclareCaptionLabelFormat{andtable}{#1~#2  \&  \tablename~\thetable}
\usepackage{enumitem}
\usepackage{booktabs}
\usepackage{wrapfig}
\usepackage{caption}
\usepackage{subcaption}
\usepackage{hyperref}

\newtheorem{theorem}{Theorem}[section]

\theoremstyle{definition}

\theoremstyle{remark}

\numberwithin{equation}{section}

\begin{document}


\title{A spacetime DPG method for \blue{the
      wave equation in multiple dimensions}}

\author{J.~Gopalakrishnan}
\address{Portland State University, Po Box 751, Portland, OR 97207-0751, USA.}
\email{gjay@pdx.edu}

\author{P. ~Sep\'ulveda}
\address{{\blue{Basque Center for Applied Mathematics, Mazarredo 14, Bilbao, E48009, Spain.}}}

\email{psepulveda@bcamath.org}

\thanks{This work was partly supported by AFOSR grant
  FA9550-17-1-0090. \blue{ Numerical studies were facilitated by the
    Portland Institute of Sciences (PICS) established under NSF grant
    DMS-1624776. Paulina Sep\'ulveda has received funding from the
    Spanish Ministry of Economy and Competitiveness with reference
    MTM2016-76329-R (AEI/FEDER, EU) and BCAM ``Severo Ochoa''
    accreditation of excellence SEV-2013-0323 and SEV-2017-0718.}}

\begin{abstract}
 A spacetime discontinuous Petrov-Galerkin (DPG) method for
  the linear wave equation is presented. \blue{This method is based on
    a weak formulation that uses a broken graph space. The
    wellposedness of this formulation is established using a
    previously presented abstract framework.} One of the main tasks in
  the verification of the conditions of this framework is proving a
  density result. This is done in detail for a simple domain in
  arbitrary dimensions. The DPG method based on the weak formulation
  is then studied theoretically and numerically.  Error estimates and
  numerical results are presented for triangular, rectangular,
  tetrahedral, and hexahedral meshes of the spacetime domain.
  \blue{The potential for using the built-in error estimator of the DPG
   method for an adaptivity mesh refinement strategy in two and three 
   dimensions is also presented.}
\end{abstract}

\maketitle

\section{Introduction}

This is a study on the feasibility of the discontinuous Petrov
Galerkin (DPG) method \cite{JG:Jay2012,JG:DemkoGopal13} for the
spacetime wave equation. We follow the approach laid out in our
earlier study of the DPG method for the spacetime Schr\"odinger
equation~\cite{JG:DGNS2017}.

Currently, the most widely used numerical techniques for transient
problems are time-stepping schemes (based on the method of lines
approach).  However, there has been increasing interest recently in
direct spacetime discretizations (where time is viewed as just another
coordinate).  Some reasons for investigating these approaches include
their potential for performing natural spacetime adaptivity,
possibility to obtain convergence even under limited spacetime
regularity, exploitation of parallelism without causality constraints,
and treatment of moving boundaries \blue{(see e.g.~\cite{JG:DGNS2017,
    JG:Neumu13, JG:NeumuellerVassilevskiVilla2017, JG:Stein15,
    JG:VoroninLeeNeumulerSepulvedaVassilevski18})}.  \blue{The
  analysis and implementation of 4D finite element discretizations is
  already underway~\cite{JG:GopalNeumuVassi18,
    JG:VoroninLeeNeumulerSepulvedaVassilevski18}, hence our interest
  in obtaining a wellposed formulation in arbitrary dimensions.}

Since the DPG method has a built-in error estimator and exhibits good
pre-asymptotic mesh-independent stability properties, it is natural to
consider its extension to spacetime problems.  Applications of the DPG
method for spacetime problems have already been computationally
studied in \cite{JG:EDCM2014} for the transient parabolic partial
differential equations and~\cite{JG:ECD2015} for the time-dependent
convection-diffusion equation. We also note that a scheme that
combines DPG spatial discretization with backward Euler time stepping
for the heat equation has been analyzed in~\cite{JG:FHG2017}.

In contrast to these works, here we consider the transient acoustic
wave system \blue{in arbitrary dimensions}.  One contribution of this
work is a proof of the wellposedness of the ultraweak DPG formulation
for the spacetime wave problem in a non-standard Hilbert space,
\blue{without developing a trace theory for this function space}. By
using the abstract theory developed in~\cite{JG:DGNS2017}, the proof
reduces to verification of some conditions. This verification proceeds
by proving a density result. \blue{The presented proof only applies
  for a multi-dimensional hyper-rectangle}.

We also present, both
practically and theoretically, how the built-in DPG error estimator is
useful for spacetime adaptive refinement \blue{ in two and three
  dimensions using conforming meshes of simplices}.  We
also show that depending on how the interfacial variables are treated,
one may end up with a discrete DPG system that has a nontrivial kernel
for some alignment of mesh facets, a difficulty that we have not
previously 
encountered in any other DPG example. We then provide
practical solutions for solving for the DPG wave approximations
despite the null space. \blue{The solutions computed using these
  techniques were observed to converge at the optimal rate.}

In Section~\ref{JGsc:model} we introduce the model wave problem and
put it into the abstract variational setting of~\cite{JG:DGNS2017}. In
Section~\ref{JG:weakform} we introduce a broken weak formulation (upon
which the DPG method is based) and prove its wellposedness subject to
a density condition. In Section~\ref{JGsc:density} we give a proof of
the density condition for a simple multi-dimen- sional domain.  In
Section~\ref{JGsc:estimates} we prove error estimates for the ideal
DPG method for solutions with enough regularity.  Finally, numerical
experiments and implementation techniques are presented in Section
\ref{JGsc:numerical}.

\section{The transient wave problem}
\label{JGsc:model}

Let $\om_0$ be a \blue{spatial domain in}  $\mathbb{R}^{d}$, with
\blue{Lipschitz} boundary $\dd \om_0$, and let
$\Omega = \Omega_0\times (0,T)$ be the spacetime domain, where $T>0$
represents the final time.  We consider the first order system for the
wave equation given by
\begin{subequations}
  \label{JG:model} 
\begin{align}
\dt q - \blue{c} \grad_x \mu = g,
\\
\dt \mu -  \blue{c} \div_x q  = f, 
\end{align}
where $f \in L^2(\om)$ and $g \in L^2(\om)^{d}$ \blue{ and $c>0$ is
  the constant wave speed}.  Here the differential operators
$\mathrm{div}_x $ and $ \grad_x$ represent the (distributional)
divergence and gradient operators that differentiate only along the
spatial components ($x$). We add homogeneous initial and boundary
conditions:
\begin{equation}
  \label{JGmodel-bc}
  \mu|_{t=0} =0, 
\qquad 
q|_{t=0} = 0, 
\qquad \mu|_{\dd\om_{\blue{0}}\times (0,T)} = 0.
\end{equation}
\end{subequations}  
Here, $q$ represents the velocity and $\mu$ the pressure.  We now cast
this problem in the framework of the abstract setting
in~\cite[Appendix~A]{JG:DGNS2017}.

\subsection{The formal wave operator}

Formally, the wave operator generated by the above system may be
considered as a first order distributional derivative operator.
Namely, set $A: L^2(\om)^{d+1} \to \D'(\om)^{d+1}$ by
\begin{equation}
\label{JGeq:A}
A u = \bb{ \dt u_q -\blue{c} \grad_x u_\mu \\ \dt u_\mu - \blue{c}\div_x u_q} 
\end{equation}
where $u$ in $L^2(\om)^{d+1}$ is block partitioned into
\begin{equation}
  \label{JGeq:block}
u = 
\begin{bmatrix}
  u_q \\ u_\mu
\end{bmatrix}, 
\qquad u_q \in L^2(\om)^d, \quad u_\mu \in L^2(\om).  
\end{equation}

Next, we introduce the space
\[
  W(\om) = \{ u \in L^2(\om)^{d+1} : A u \in L^2(\om)^{d+1} \}\blue{.}
\]
By $W(K)$ we mean the similarly defined space on an open subset $K$ of
$\om$, but when considering this space with domain $\om$, we
abbreviate $W(\om) = W$.  Hereon, we denote by $(\cdot, \cdot)$ and
$\|\cdot\|$ the $L^2(\om)^{d+1}$ inner product and norm, 
respectively\blue{, and $\D(\om)^{d+1}$ and $\D'(\om)^{d+1}$ is the space of  infinitely differentiable vector functions with compact support in $\om$ and its dual space, respectively}.
It is well known that the space $W(\om)$ is a Hilbert space when endowed
with the graph norm $\|u \|_W = (\| u\|^2 + \|A u\|^2)^{1/2}$ (see
\cite[Lemma A.1.]{JG:DGNS2017}).  The formal adjoint of $A$ is the
distributional differentiation operator $-A$ and \blue{it} satisfies
\begin{align*}
(Aw, \wt) = -(w, A\wt) \qquad \text{ for all } w, \wt \in \D(\om)^{d+1}.
\end{align*}
\blue{Define the  operator $D: W \to W'$ by}
\begin{equation}
\ip{ D u, v}_W
= ( A u, v)_\om + (u, Av)_\om \qquad \text{for all } u, v \in W. \label{JGdef:D}
\end{equation}
Here $W'$ is the dual space of $W$, and $\ip{\cdot, \cdot}_W$
represents the duality pairing of a \blue{functional in $W'$ with an element
of $W$.} For smooth functions $u, v \in \D(\bar \om)^{d+1}$,
integration by parts shows that
\begin{equation}
  \ip{D u, v}_W = \int_{\dd\om} u_q \cdot (n_t v_q - \blue{c} n_x v_\mu) + u_\mu ( n_t v_\mu -  \blue{c}  n_x \cdot v_q ). \label{JGbdryop}
\end{equation}
Here and throughout, $n =\blue{ (n_x^\texttt{T},n_t)^\texttt{T}}$ represents the unit outward
normal component to $\om$ in $\mathbb{R}^{d+1}$ and functions in
$L^2(\om)^{d+1}$, like the $u$ and $v$ above, are block partitioned as
in~\eqref{JGeq:block}.

\subsection{The unbounded wave operator}

In order to consider the boundary and initial conditions, we now
proceed as suggested in~\cite[Appendix~A]{JG:DGNS2017}, to define an
unbounded operator with a domain that takes these conditions into
account. Below, by abusing the notation, we shall denote this
unbounded operator also by~$A$.

First, let us partition the spacetime boundary $\dd\om$ into
\[
\Gamma_0 = \om_{0} \times\{0\},
\qquad 
\Gamma_T = \om_0\times \{ T \}, 
\qquad 
\Gamma_b = \partial \om_0\times[0,T].
\]
We define the following sets of smooth functions:
\begin{align}
\V &= \{ u\in \D(\bar \om)^{d+1}: u|_{\Gamma_0} = 0, u_\mu|_{\Gamma_b} =0 \}, \label{JGdef:cV} \\
\V^* & = \{ v \in \D (\bar \om) ^{d+1}: v |_{\Gamma_T} = 0, v_\mu|_{\Gamma_b} =0 \}. \label{JGdef:cV*}
\end{align}
Next, let $A: \dom(A) \subset L^2(\om)^{d+1} \to L^2(\om)^{d+1}$ be the unbounded
operator in $L^2(\om)^{d+1}$ defined by the right hand side
of~\eqref{JGeq:A} with 
\begin{equation}
\dom(A) = \{ u \in W: \ip{ D u, v}_W = 0 \text{ for all } v \in
\V^*\}. \label{JGdef:V}
\end{equation}
From~\eqref{JGbdryop}, we see that the set of smooth functions
$\D(\om)^{d+1}$ is contained in $\dom(A).$ Hence, $A$ is a densely
defined operator in $L^2(\om)^{d+1}$. Therefore, it has a uniquely
defined adjoint $A^*$, which is again an unbounded operator. The
adjoint $A^*$ equals the distributional derivative operator $-A$ when
applied to $\dom(A^*)$. This domain is prescribed as in standard
functional analysis~\cite{JG:Brezis11} by
\begin{align*}
\dom(A^*) = 
\big\{ \wt \in L^2(\om)^{d+1}: \;
 \exists\, \ell \in L^2(\om)^{d+1}
   \text{ such that } &(Au, \wt) =(u,\ell)\\
  & \quad \text{ for all } u \in \dom( A)\big\}.
\end{align*}

By definition, $\dom(A)$ is a subset of $W(\om)$.  When this subset is
given the topology of $W(\om)$, we obtain a closed subset of $W(\om)$,
which we call $V$, i.e., $V$ and $\dom(A)$ coincide as sets or vector
spaces, but not as topological spaces. Note that $\dom(A^*)$ is also a
subset of $W$, since for any $\tilde w \in \dom(A^*)$, the
distribution $-A\tilde w$ is in $L^2(\om)^{\blue{d+1}}$.  When $\dom(A^*)$ is given
the topology of $W$, it will be denoted by $V^*$. Observe that since
$V$ is closed, $A$ is a closed operator. 
\blue{ For any $S\subset W$ subspace, the right annihilator of $S$, denote by $\lprp S$, is  defined by
\begin{align}
\lprp S = \{w \in W : \ip{s',w}_W = 0 \text{ for all } s' \in  S\}.
\end{align}}
The definition of $\dom(A^*),$ when written in terms of $D$ reveals  that
 \begin{equation} 
V^* = \blue{\lprp {D(V)}}.  \label{JGperpv*}
 \end{equation}
 Thus $V^*$ is also a closed subset of $W$.

The next observation is that from the definitions of $V$ and the
operator $D$ (namely~\eqref{JGdef:V} and~\eqref{JGdef:D}) it
immediately follows that $\V\subset V$. Note also that if
$v^* \in \V^*$, then $(Au, v^*)=-(u, Av^*) + \ip{Du, \blue{v^*}}_W =-(u, Av^*)$
for all $u\in V$, since $\ip{Du,\blue{v^*}}_W = 0 $ by the definition of
$V$. Therefore $v^*$ is in $V^*$. \blue{To summarize these
  observations, we have introduced $\V, \V^*, V$ and $V^*,$ satisfying}
\begin{equation}
  \label{JGeq:contained}
  \V\subset V \quad\text{ and } \quad \V^* \subset V^*.  
\end{equation}
\blue{These are the abstract ingredients
in the framework of}~\cite[Appendix~A]{JG:DGNS2017} applied to the wave problem.

\section{The broken weak formulation}\label{JG:weakform}
Following the settings of \cite{JG:CarstDemkoGopal16} and
\cite[Appendix]{JG:DGNS2017}, we partition the spacetime \blue{Lipschitz} domain $\om$
into a mesh $\om_h$ of finitely many open elements $K$, \blue{(e.g. $(d+1)$-simplices or $(d+1)$-hyperrectangles)}  such that
$\bar \om = \cup_{K \in \om_h} \bar K$ where
$h=\max_{K \in \om_h} \mathrm{diam}(K)$.  The DPG method is based on a
variational formulation in a ``broken'' analogue of $W$, which we call
$W_h$, defined below.

We let $A_h$ be the wave operator applied element by element, i.e.,
\[
(A_h w)|_K = A(w|_K), \qquad w \in W(K), \quad K \in \om_h.
\]
Let $W_h = \{ w \in L^2(\om)^{d+1}: A_h w \in L^2(\om)^{d+1}\}$.  The
operator 
$
D_h : W_h \to W_h'$ is defined by
\[
\ip{D_h w, v}_{W_h} = (A_h w,v)_\om + (w, A_h v)_\om
\]
for all $w,v \in W_h$, where $\ip{\cdot,\cdot}_{W_h}$ denotes the
duality pairing in $W_h$ in accordance with our previous notation.
Below we abbreviate $\ip{\cdot,\cdot}_{W_h}$ to
$\ip{ \cdot, \cdot}_h$.  Let $D_{h,V}: V \to W_h'$ denote the
restriction of $D_h$ to $V$, i.e., $D_{h,V} = D_h|_V$.  The range of
$D_{h,V}$, denoted by $Q,$ is made into a complete space by the
quotient norm
\begin{align}
\| \rho \|_Q = \inf_{ v \in D_{h,V}^{-1}(\{ \rho \}) } \| v \|_W, 
  \qquad \rho \in Q \equiv \text{ran}(D_{h,V}). \label{JGQnorm}
\end{align}

Define the bilinear form on $(L^2(\om)^{d+1} \times Q) \times W_h$ by
$$ 
  b( (v,\rho), w)  =- (v, A_h w)_\om + \ip{ \rho, w }_h.
$$ 
The ``broken'' variational formulation for the wave problem now reads
as follows. Given any $F$ in the dual space $W_h'$, find
$u\in L^2(\om)^{d+1}$ and $ \lambda \in Q$ such that
\begin{align}
b( (u,\lambda), w) = F(w) \qquad \text{ for all }  w \in W_h.
  \label{JGproblem}
\end{align}
Critical to the success of any numerical approximation of this
formulation, in particular, the DPG approximation, is its
wellposedness.  By~\cite[Theorem A.5]{JG:DGNS2017}, this formulation
is well-posed, provided we verify
\begin{align}
  V &= \lprp{D(V^*)},  \label{JGth:V1} 
  \\
  A: V \to & L^2(\om)^{d+1} \text{ is a bijection.} \label{JGth:V2}
\end{align}
Therefore our next focus is on proving~\eqref{JGth:V1}
and~\eqref{JGth:V2}. Recall from~\eqref{JGeq:contained} that $\V$ and
$\V^*$ are subspaces of smooth functions within $V$ and $V^*$. 
We now
show that~\eqref{JGth:V1} and~\eqref{JGth:V2} follow if these are
dense subspaces.

\begin{theorem}\label{JGthm:wellposedness}
  Suppose 
  \begin{equation}
    \label{JGeq:dense}
    \text{$\V$ is dense in $V$ and $\V^*$ is dense $V^*.$}
  \end{equation}
  Then~\eqref{JGth:V1} and~\eqref{JGth:V2} holds. Consequently, the
  broken weak formulation~\eqref{JGproblem} is well posed.
\end{theorem}
\begin{proof}
  In view of the continuity of $D$, ~\eqref{JGeq:contained}, and the
  assumption that $\V^*$ is dense in $V^*,$ the
  condition~\eqref{JGth:V1} now immediately follows.

  Next, we will prove that 
  \begin{subequations}
    \label{JGeq:bddbelowsmooth}
    \begin{align}
    \|u\| \leq 2T \| A u\|,  \qquad \text{ for all } u \in \V ,\\
    \|v\| \leq 2T \| A^* v\| , \qquad \text{ for all } v \in \V^*.
  \end{align}  
\end{subequations}
  These inequalities follow by well-known energy arguments, as shown
  in~\cite[Lemma~3]{JG:Wieners16}. We briefly include the proof for
  completeness. Let $v\in \V^*$. Then
  \begin{align*}
    \| v\|^2
    &
      = \int_0^T\left( \int_{\om_0} |v(x,t)|^2 \;dx \right) dt
      =
      2\int_0^T \!\! \int_T^t \!\! \int_{\om_0} \partial_s  v(x,s)\cdot v(x,s) \,dx\,ds\,dt
    \\
    & = 2 \int_0^T\!\! \int_t^T\!\! \int_{\om_0} v(x,s) \cdot A^* v(x,s) \, dx\, ds\, dt -2\blue{c} \int_0 ^T\!\!
    \int_t^T \!\!  \int_{\partial\om_0}   (v_q\cdot n_x)v_ \mu \,
      dx \, ds \, dt 
    \\
    & \leq 2T \| v\|\; \|A^* v\|.
  \end{align*}
  The inequality for $\V$ is similarly proved by using its boundary
  conditions instead of those of $\V^*$. 

  Using the density assumptions, we conclude
  that~\eqref{JGeq:bddbelowsmooth} implies 
  \begin{subequations}
  \begin{align}
    \| u \| \le 2T \|Au \|
    & \qquad \text{ for all } u \in V \text{
      and} \label{JGbdb:V}
    \\
    \| v \| \le 2T \|A^*v\|&\qquad \text{ for all } v \in V^*.
                          \label{JGbdb:V*}
  \end{align}
  \end{subequations}
  The inequality~\eqref{JGbdb:V} and the closed range theorem for
  closed operators imply that $A: \dom(A)= V \to L^2(\om)^{d+1}$ is
  injective and has closed range. Moreover, its adjoint $A^*$ is
  injective (on its domain) by \eqref{JGbdb:V*}, so the range of $A$
  must be all of $L^2(\om)^{d+1}$ (see e.g., \cite[Corollary
  2.18]{JG:Brezis11}).  Hence~$A$ is a bijection, i.e.,
  condition~\eqref{JGth:V2} holds.  Finally, since we have verified
  both~\eqref{JGth:V1} and~\eqref{JGth:V2}, applying \cite[Theorem
  A.5]{JG:DGNS2017}, the wellposedness {follows}.
\end{proof}

Note that the wellposedness result of
Theorem~\ref{JGthm:wellposedness}, in particular, implies that
\begin{equation}
  \label{JGeq:beta}
\beta = \;\inf_{0 \ne (v, \rho) \in L^2(\om)^{d+1}\times Q}\;
 \sup_{ 0 \ne w \in W_h } \frac{ b((v,\rho), w)}{ \| (v,
   \rho)\|_{L^2(\om)^{d+1} \times Q}\blue{ \| w \|_{W_h}}} >0.
\end{equation}

\section{Verification of the density condition} \label{JGsc:density} 

In this section, we verify~\eqref{JGeq:dense} for a simple domain,
namely a hyperrectangle (or an orthotope). Accordingly, throughout
this section, we fix $\om = \om_0\times (0,T)$ and
\[
\om_0 = \prod_{i=1}^{d} (0,a_i),
\]
for some $a_i>0$.  \blue{While density of smooth functions in general
  graph spaces can be proved by standard Sobolev space
  techniques~\cite{JG:Adams75}, to obtain the density of smooth
  functions with boundary conditions (like those in $\V$) we need more
  arguments.}  The following proof has some similarities with the
proof of \cite[Theorem~3.1]{JG:DGNS2017}, an analogous density result
for the one-dimensional Schr\"odinger operator. The main differences
from~\cite{JG:DGNS2017} in the proof below include the consideration
of multiple spatial dimensions and the construction of extension
operators for vector functions in the wave graph space by combining
even and odd reflections appropriately.

\begin{theorem}
  On the above $\om$, $\V^*$ is dense in $V^*$ and $\V$ is dense in
  $V$.
\end{theorem}
\begin{proof}  
  We shall only prove that $\V$ is dense in $V$, since the proof of
  the density of $\V^*$ in $V^*$ is similar.
  We divide the proof into three main steps. 

  {\em Step 1.  Extension:} In this step, we will extend a function in
  $V$ using spatial reflections to a domain which has larger spatial
  extent than $\om$ (see Figure~\ref{JGfigtrans}).

Let $e_i$ denote the standard unit basis
  vectors in $\RRR^{d+1}$ and \blue{$y \in \RRR^{d+1}$ arbitrary}. The following operations
\[
R_{i,-} \blue{y} = \blue{y}  - 2\blue{y}_i e_i, \qquad 
R_{i,+} \blue{y} =\blue{y}  + 2(a_i -\blue{y}_i) e_i
\]
perform reflections of \blue{the coordinate vector \blue{$y$ about $y_i=0$
and $y_i=a_i$,} for $i=1, \ldots, d$}. We set $Q_0 \equiv \overline \om$
and then define extended domains $Q_i$ in a recursive way, starting
from $i=1$ through $i=d$ as follows.
\[
Q_{i,-} = R_{i,-}^{-1} Q_{i-1}, \qquad Q_{i,+} = R_{i,+}^{-1} Q_{i-1}, \quad
Q_i = Q_{i,-} \cup Q_{i-1} \cup Q_{i,+}.
\]
The final extended domain is $Q \equiv Q_d.$

Next, we introduce even and odd extensions (in the $x_i$-direction) of
scalar functions. Namely, let
$G_{i,e}, G_{i,o}: L^2(Q_{i-1}) \to L^2(Q_i)$ be defined by 
 \begin{align} \label{JGGe}
 G_{i,e} f(x,t ) =  \left\{ \begin{array}{ll}
                              f( R_{i,-}\blue{(x,t)}) & \text{ if } (x,t)  \in Q_{i,-} ,  \\
                              f( R_{i,+}\blue{(x,t)}) & \text{ if }    (x,t) \in Q_{i,+}, \\
                              f(x,t) & \text{ if }    (x,t) \in Q_{i-1},
    \end{array}\right.                                 
\end{align}
and
\begin{align}\label{JGGo}
G_{i,o} f ( x, t) = \left\{ \begin{array}{ll}
                                    -f(R_{i,-} \blue{(x,t)}) & \text{ if }   (x,t)  \in Q_{i,-}, \\
                                     -f(R_{i,+}  \blue{(x,t)})& \text{ if }   (x,t) \in Q_{i,+},\\
                               f(x,t) & \text{ if }    (x,t) \in Q_{i-1}.
\end{array}\right.
\end{align}
In the case of a vector function $v \in L^2(Q_{i-1})^{d+1}$, we define
$ G_i v(x,t)$ to be the extended vector function obtained by extending
(in the $i$th direction) all the components of $v$ using the odd
scalar extension, except the $i$th component, which is extended using
the even scalar extension.  In other words, for any $i=1,\ldots, d$,
we define ${G_i} : L^2(Q_{i-1})^{d+1} \to L^2(Q_i)^{d+1}$ by
\begin{equation}
\label{JGextended}
  G_i v = (G_{i,e} v_i) e_i +   \sum_{j\ne i} (G_{i,o}v_j) e_j
\end{equation}
where the sum runs over all $j=1, \ldots, d+1$ except $i$.  Let
$E_k = G_k \circ G_{k-1} \circ \ldots \circ G_1$.  The cumulative
extension over all spatial directions is thus obtained using $E =
E_d$. It extends functions in $\overline\om$ to $Q$.

By change of variable formula for integration, we obtain 
\begin{align*}
(G_{i,o} f, g)_{Q_i} = (f, G'_{i,o} g)_{Q_{i-1}} , \qquad \text{ for all } f \in L^2(Q_{i-1}), 
\quad g \in L^2(Q_{i}),
\\
(G_{i,e} f, g)_{Q_i} = (f, G'_{i,e} g)_{Q_{i-1}} , \qquad \text{ for all } f \in L^2(Q_{i-1}), \quad 
g \in L^2(Q_{i}),
\end{align*}
where the ``folding'' operators
$G_{i,e/o}': L^2(Q_i) \to L^2(Q_{i-1}),$ that go the reverse direction
of the extension operators, are defined by
\begin{align}
G'_{i,o} g (x,t)  = g(x,t) - g(R_{i,-}^{-1}  \blue{(x,t)}) - g( R_{i,+}^{-1}  \blue{(x,t)}), \label{JGGeop}\\
G'_{i,e} g (x,t)  = g(x,t) +g(R_{i,-}^{-1} \blue{(x,t)}) + g( R_{i,+}^{-1} \blue{(x,t)}).
\end{align}
These scalar folding  operators combine to form an analogue for vector
functions as in~\eqref{JGextended}, namely
\begin{align*}
  G_i' w = (G_{i,e}' w_i) e_i +   \sum_{j\ne i} (G_{i,o}'w_j) e_j.
\end{align*}
It satisfies $ (G_i v, w )_{Q_i} = (v , G'_i w)_{Q_{i-1}} $ for all
$ v \in L^2(Q_{i-1})^{d+1}$, $w \in L^2(Q_i)^{d+1},$ and for each $i$
from $1$ to $d$. Let $E_k' = G_k' \circ G_{k+1}' \cdots \circ
G_d'$. Then $E' = E_1'$ folds functions in $Q$ to $\overline\om$ and
is the adjoint of the extension $E$ in the following sense.
\begin{equation}
(E u, w )_{Q} = (u, E' w)_{\om} , \qquad \text{ for all }
 u \in L^2(\om)^{d+1},\qquad 
w \in L^2(Q)^{d+1}. \label{JGGinv}
\end{equation}

We want to prove that $E v$ is in $W(Q)$ for any $v \in V$.  Note that
if $w \in L^2(\om)^{d+1}$, then $Ew$ in $L^2(Q)^{{d+1}}$, because each $G_i$
maps $L^2$ functions into $L^2$ per \eqref{JGextended}.
Therefore, in order to prove $Ev $ is in $W(Q)$, it only remains to
prove that $A({E}v)$ is in $L^2(Q)^{d+1}.$  Let
$\varphi \in \D(Q)^{d+1}$ (where we abuse the notation and write
$\D(Q)$ for $\D(Q^0)$ whenever $Q^0$ is the interior of $Q$).  Using
\eqref{JGGinv}, the action of the distribution $AEv$ on $\varphi$
equals
\begin{equation}
\label{JG:eqAE}
\ip{ A E v, \varphi}_{\D(Q)^{d+1}} = 
- (Ev, A\varphi)_{Q}  = -  ( v, E' A \varphi)_{\om}.
\end{equation}

To analyze the last term, first observe that by the chain rule applied
to a smooth scalar function $\phi$ on $Q_1$,
\begin{gather*}
   \dt (G'_{i,o} \phi) = G'_{i,o} \dt \phi, \quad 
   \partial_i(G_{i,o}'\phi) = G_{i,e}' (\partial_i \phi), 
   \quad 
   \partial_j(G_{i,o}'\phi) = G_{i,o}' (\partial_j \phi), 
\\  
   \dt (G'_{i,e} \phi) = G'_{i,e} \dt \phi, \quad 
   \partial_i(G_{i,e}'\phi) = G_{i,o}' (\partial_i \phi), 
   \quad 
   \partial_j(G_{i,e}'\phi) = G_{i,e}' (\partial_j \phi), 
\end{gather*}
for all $j \ne i$. Combining these appropriately for smooth vector
function $\psi$  on $Q_i$, we find that
\[
  \dt (G'_i \psi) = G'_i \dt \psi,
  \quad 
  \blue{c}\nabla_x G' _{i,o} \psi_\mu = G'_{i} \blue{(c\nabla_{x} \psi_\mu)},
  \quad 
  \blue{c}\,\mathrm{div}_x G'_{i}\psi_q = G'_{i,o} \blue{(c\,\mathrm{div}_x \, \psi_q)}.
\]
Thus, for any $\varphi \in \D(Q)^{d+1}$ we have
$ E_{i}'A \varphi = A E_i' \varphi $ for all $ i = 1, \cdots, d, $ and
in particular
\begin{equation}
  \label{JGeq:2}
E' A \varphi = A E' \varphi.   
\end{equation}
Returning to~\eqref{JG:eqAE} and using~\eqref{JGeq:2} and \eqref{JGdef:D},
\begin{equation}
\label{JG:eq:1}
\ip{ A  E v, \varphi}_{\D(Q)^{d+1}} = (Av,  E' \varphi)_\om- \ip{ D v, E' \varphi} _{W(\om)}.
\end{equation}

We shall now show that the last term above vanishes. Since $v$ is in
$V$, the last term will vanish by the definition of $V$, provided
$E'\varphi$ is in $\V^*$. To prove that $E'\varphi$ is in $\V^*$, we
only need to verify that $E'\varphi$ satisfies the boundary conditions
in~\eqref{JGdef:cV*}.  Since $ \varphi$ is compactly supported in $Q$,
we obviously have $( E'\varphi)|_{\Gamma_T} = 0$ as $E'$ only involves
spatial folding.  

Next, we claim that $[E'\varphi]_{\mu} |_{\Gamma_b}= 0$ also.  To see
this, let $\Gamma^j$ denote the two \blue{facets} of $\partial Q_{j}$ where
$x_j$ is constant and $\gamma^j$ denote the two \blue{facets} of
$\partial Q_{j-1}$ where $x_j$ is constant.  The value of
$G_{d,o}' \varphi_\mu(x,t)$ for any $(x,t)$ in $\gamma^{d-1}$ is the
sum of the three terms in~\eqref{JGGeop}, two of which cancel each
other, and one of which vanishes because
$\varphi_\mu|_{\Gamma^d}=0$. Thus
$\varphi_\mu|_{\Gamma^d}=0 \implies (G_{d,o}'
\varphi_\mu)\big|_{\partial Q_{d-1}}=0$ (where we have also used the
fact that $\varphi_\mu$ vanishes on the remainder
$\partial Q_{d-1} \setminus \gamma^{d-1}$).  The same argument can now
be repeated to get that
$(G_{d,o}' \varphi_\mu)\big|_{\Gamma^{d-1}}=0 \implies \left(G_{d-1,o}
  ( G_{d,o} \varphi_\mu)\right)\big|_{\partial Q_{d-2}} = 0.  $
Continuing similarly, we obtain that
$ [E'\varphi]_{\mu} = G_{1,o}' \circ G_{2,o}'\circ \cdots \circ
G_{d,o}' \varphi_\mu$ vanishes on $ \partial Q_0 = \Gamma_b$. Thus, the
last term in~\eqref{JG:eq:1} is zero and by~\eqref{JGGinv} we conclude
that
\begin{equation}
  \label{JGeq:AE-EA}
  \ip{ A  E v, \varphi}_{\D(Q)^{d+1}} = (EAv,  \varphi)_Q
\end{equation}
for all $\varphi$ in $\D(Q)^{d+1}$. 

By virtue of~\eqref{JGeq:AE-EA}, we have proved that for any
$v \in V,$ $AE v$ is in $L^2(Q)^{d+1},$ $AE v$ coincides with $EAv$,
and $Ev$ is in $W(Q)$.

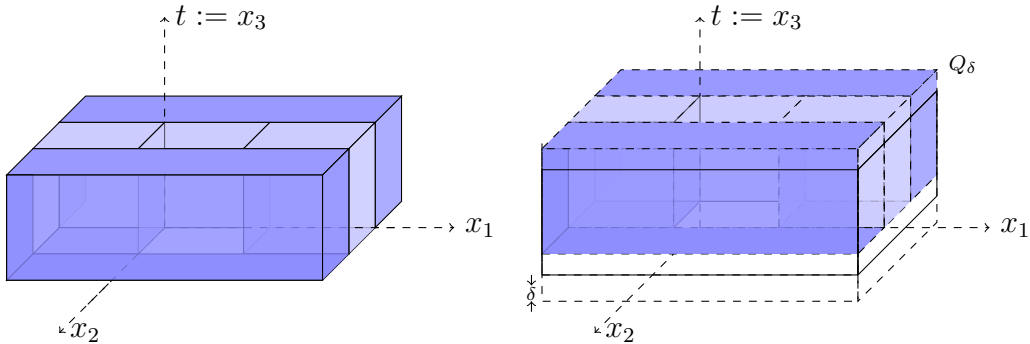
\begin{figure}[b]
\begin{tikzpicture}[scale=0.70]
\begin{scope}[shift={(0,-.5)}]
\begin{scope}[shift={(-.5,-.5)}] 

\begin{scope}[shift={(1,1)}]
\draw[fill=blue!25, fill opacity=0.8] (-3,0)--(-3,2) -- (-2.5,2) -- (-2.5,0.5) --cycle;
\draw[fill=blue!10,fill opacity=0.8] (-3,0)--(3,0)-- (3,2) --(-3,2)--cycle;
\draw[fill=blue!40,fill opacity=0.8] (3.5,0.5) --(3.5,2.5)--(3,2) -- (3,0)--cycle;
\end{scope}

\begin{scope}[shift={(1,3)}]
\draw[fill=blue!40]  (-3,0)--(3,0) --(3.5,0.5) --(-2.5,0.5)--cycle;
\end{scope}

\end{scope}
 \draw[->,dashed] (-0.5,0.5)--(-0.5,4.5) [right]node{ $t:=x_3$};
 \draw[->,dashed] (-0.5,0.5)--(5,0.5) [right]node{ $x_1$};
 \draw[->,dashed] (-0.5,0.5)--(-2.5,-1.5) [right]node{ $x_2$};


\begin{scope} 
\begin{scope}
\end{scope}
\begin{scope}[shift={(0,2)}]
\draw[fill=blue!20] (-1,0)--(1,0) --(1.5,0.5) --(-0.5,0.5)--cycle;

\end{scope}

%
\end{scope}


\begin{scope}[shift={(-2,0)}] 
\begin{scope}
\draw[fill=blue!20] (-1,0)--(1,0) --(1.5,0.5) --(-0.5,0.5)--cycle;
\end{scope}
\begin{scope}[shift={(0,2)}]
\draw[fill=blue!20] (-1,0)--(1,0) --(1.5,0.5) --(-0.5,0.5)--cycle;
\end{scope}

\begin{scope}
\draw[fill=blue!20,fill opacity=0.3,opacity=0.6] (1.5,0.5) --(1.5,2.5)--(1,2) -- (1,0);

\end{scope}
\end{scope}
\begin{scope}[shift={(+2,0)}] 
\begin{scope}
\draw[fill=blue!20] (-1,0)--(1,0) --(1.5,0.5) --(-0.5,0.5)--cycle;
\end{scope}
\begin{scope}[shift={(0,2)}]
\draw[fill=blue!20] (-1,0)--(1,0) --(1.5,0.5) --(-0.5,0.5)--cycle;
\end{scope}

\begin{scope}
\draw[fill=blue!15, fill opacity=0.8] (-0.5,0.5)--(-0.5,1.5) -- (-1,2) -- (-1,0) --cycle;
\draw[fill=blue!20,fill opacity=0.8] (-1,0)--(1,0)-- (1,2) --(-1,2)--cycle;
\draw[fill=blue!20,fill opacity=0.8] (1.5,0.5) --(1.5,2.5)--(1,2) -- (1,0);

\end{scope}
\end{scope}
\begin{scope}[shift={(-.5,-.5)}] 
\begin{scope}
\draw[fill=blue!40] (-3,0)--(3,0) --(3.5,0.5) --(-2.5,0.5)--cycle;
\end{scope}
\begin{scope}[shift={(0,2)}]
\draw[fill=blue!40]  (-3,0)--(3,0) --(3.5,0.5) --(-2.5,0.5)--cycle;
\end{scope}

 \draw[-,dashed] (-0,2)--(0,2);
 \draw[-,dashed] (-0,1)--(4.5,1);
 \draw[-,dashed] (-0.5,0.5)--(-1.5,-.5);

\begin{scope}
\draw[fill=blue!45, fill opacity=0.8] (-3,0)--(-3,2) -- (-2.5,2) -- (-2.5,0.5) --cycle;
\draw[fill=blue!45,fill opacity=0.8] (-3,0)--(3,0)-- (3,2) --(-3,2)--cycle;
\draw[fill=blue!45,fill opacity=0.8] (3.5,0.5) --(3.5,2.5)--(3,2) -- (3,0);
\end{scope}

\end{scope}
\end{scope}
\end{tikzpicture}
\begin{tikzpicture}[scale=0.70,shift={(0,3)}]
\begin{scope}[shift={(0,8)}]
\begin{scope}[shift={(-.5,-.5)}] 
\begin{scope}[shift={(1,1)}]
\draw[fill=blue!25, fill opacity=0.8] (-3,0)--(-3,2) -- (-2.5,2) -- (-2.5,0.5) --cycle;
\draw[fill=blue!10,fill opacity=0.8] (-3,0)--(3,0)-- (3,2) --(-3,2)--cycle;
\draw[dashed,fill=blue!40,fill opacity=0.8] (3.5,0.5) --(3.5,2.5)--(3,2) -- (3,0)--cycle;
\end{scope}

\begin{scope}[shift={(1,3)}]
\draw[dashed,fill=blue!40]  (-3,0)--(3,0) --(3.5,0.5) --(-2.5,0.5)--cycle;
\end{scope}

\end{scope}
\begin{scope}[shift={(0,-.5)}]
 \draw[->,dashed] (-0.5,0.5)--(-0.5,4.5) [right]node{ $t:=x_3$};
 \draw[->,dashed] (-0.5,0.5)--(5,0.5) [right]node{ $x_1$};
 \draw[->,dashed] (-0.5,0.5)--(-2.5,-1.5) [right]node{ $x_2$};
\end{scope}

\begin{scope} 
\begin{scope}
\end{scope}
\begin{scope}[shift={(0,2)}]
\draw[dashed,fill=blue!20] (-1,0)--(1,0) --(1.5,0.5) --(-0.5,0.5)--cycle;
\end{scope}
\end{scope}


\begin{scope}[shift={(-2,0)}] 
\begin{scope}
\draw[fill=blue!20] (-1,0)--(1,0) --(1.5,0.5) --(-0.5,0.5)--cycle;
\end{scope}
\begin{scope}[shift={(0,2)}]
\draw[dashed,fill=blue!20] (-1,0)--(1,0) --(1.5,0.5) --(-0.5,0.5)--cycle;
\end{scope}

\begin{scope}
  \draw[fill=blue!15, fill opacity=0.8] (-0.5,0.5)--(-0.5,1.5) -- (-1,2) -- (-1,0) --cycle;
  \draw[dashed,fill=blue!20,fill opacity=0.8] (-1,0)--(1,0)-- (1,2) --(-1,2)--cycle;
  \draw[,fill=blue!20,fill opacity=0.3,opacity=0.6] (1.5,0.5) --(1.5,2.5)--(1,2) -- (1,0);
\end{scope}
\end{scope}
\begin{scope}[shift={(+2,0)}] 
\begin{scope}
\draw[dashed,fill=blue!20] (-1,0)--(1,0) --(1.5,0.5) --(-0.5,0.5)--cycle;
\end{scope}
\begin{scope}[shift={(0,2)}]
  \draw[dashed,fill=blue!20] (-1,0)--(1,0) --(1.5,0.5) --(-0.5,0.5)--cycle;
\end{scope}

\begin{scope}
  \draw[dashed,fill=blue!15, fill opacity=0.8] (-0.5,0.5)--(-0.5,1.5) -- (-1,2) -- (-1,0) --cycle;
  \draw[dashed,fill=blue!20,fill opacity=0.8] (-1,0)--(1,0)-- (1,2) --(-1,2)--cycle;
  \draw[dashed,fill=blue!20,fill opacity=0.8] (1.5,0.5) --(1.5,2.5)--(1,2) -- (1,0);

\end{scope}
\end{scope}
\begin{scope}[shift={(-.5,-.5)}] 
  \begin{scope}
    \draw[dashed,fill=blue!40] (-3,0)--(3,0) --(3.5,0.5) --(-2.5,0.5)--cycle;
  \end{scope}
  \begin{scope}[shift={(0,2)}]
    \draw[dashed,fill=blue!40]  (-3,0)--(3,0) --(3.5,0.5) --(-2.5,0.5)--cycle;
  \end{scope}


  \begin{scope}
    \draw[fill=blue!45, fill opacity=0.8] (-3,0)--(-3,2) -- (-2.5,2) -- (-2.5,0.5) --cycle;
    \fill[blue!45, opacity=0.8] (-3,0)--(3,0)-- (3,2) --(-3,2)--cycle;
    \draw[dashed] (-3,-.4)--(3,-.4)-- (3,2) --(-3,2)--cycle;
    \draw[dashed] (3.5,0.5) --(3.5,2.5)--(3,2) -- (3,0);
  \end{scope}

\end{scope}


\begin{scope}[shift={(-.5,-.9)}] 
\begin{scope}
\end{scope}
\begin{scope}[shift={(0,2)}]
\end{scope}

\begin{scope}
\draw[] (-3,0)--(3,0)-- (3,2) --(-3,2)--cycle; 
\draw[](4.5,1.5) --(4.5,3.5)--(3,2) -- (3,0)--cycle; 
\draw[dashed] (-3,2.5) -- (-3,-.5)--(3,-.5)-- (3,2.5) ; 
\draw[dashed] (3,2) -- (3,-.5) -- (4.5,1) --(4.5,3.5); 
\draw[] (4.5,3.5)--(3,2) ;
\end{scope}
\begin{scope}[shift={(-8.2,-2.3)}]
  \draw[->] (5,2.1+0.2)--(5,2.1);
          \draw[->] (5,2.1-0.5)--(5,2.1-0.3);
          \node at  (5,2.1-0.15) {\fontsize{7}{15}\selectfont$\delta$};
\end{scope}
\node at  (5,4) {\fontsize{8}{15}\selectfont$Q_\delta$};

\end{scope}
\end{scope}
\end{tikzpicture}
\caption{{\em Left:} Extended domains $Q_1$ and $Q_2$ when $\om
  \subseteq \RRR^3$.
{\em Right:} Translation by $\delta$ in the $t$ direction.
\label{JGfigtrans}}
\end{figure}

{\em Step 2. Translation:} In this step, we will translate up the
previously obtained extension in time coordinate. This will give us
room to mollify in the next step. Such a translation step is standard
in many density proofs (see e.g., \cite{JG:Adams75}).

Let $v \in V$ and let $\tilde E v $ denote the extension of $E v$ by
zero to $\mathbb{R}^{d+1},$ i.e., $\tilde E v$ equals $E v $ in $ Q$
and it is zero elsewhere. Denote by $\tau_\delta$ the translation
operator in the $t$-direction by $\delta >0$; i.e.
$(\tau_\delta w) (x,t) = w(x,t-\delta)$ for scalar or vector
functions~$w$. It is well known \cite{JG:Brezis11} that
\begin{align}
  \lim_{\delta \to 0} \|\tau_\delta g - g\|_{L^2(\mathbb{R}^{d+1})}= 0 , \quad 
\forall g \in L^2(\mathbb{R}^{d+1}). 
  \label{JGtranslation}
\end{align}
Let
$Q_\delta = \prod_{ i=1}^{i=d}(-a_i, 2a_i)\times (-\delta, T +
\delta)$ and let $H_\delta$ be the restriction from
$\mathbb{R}^{d+1} $ to~$Q_\delta$.

We will now show that
\begin{align}
  \label{JGeq:AHtE}
 A H_\delta \tau _\delta \tilde E v = H_\delta \tau_\delta \tilde E A v.
\end{align}
By a change of variable,
\begin{align} \label{JGtrans}
(\tau_\delta \tilde E w, \tilde w) _{Q_\delta}  = (E w, \tau_{-\delta} \tilde w)  _ Q
\end{align}
for all $ w\in L^2(\om)^{d+1}$ and
$ \tilde w \in L^2(Q_\delta)^{d+1}.$ Note that
$\tau_\delta \tilde E v|_{\om} \in L^2(\om)^{ \blue{d+1} }$.  The distribution
$A H_\delta \tau _\delta \tilde E v $ applied to a smooth function
$\varphi \in \D(Q_\delta)^{{d+1}} $ equals
\[
 \ip{ A H_\delta \tau_\delta \tilde E v, \varphi}_{\D(Q_\delta)^{d+1}}  =
- (\tau_\delta \tilde E v, A \varphi)_{Q_\delta} = -( Ev, A
\tau_{-\delta} \varphi)_{Q} 
\]
due to \eqref{JGtrans} and \blue{ the fact that}
$\tau_{-\delta} A\varphi = A\tau_{-\delta} \varphi$.
Using also \eqref{JGGinv} and \eqref{JGeq:2}, 
\begin{align*}
  \ip{ A H_\delta \tau_\delta \tilde E v, \varphi}_{\D(Q_\delta)^{d+1}}  
  & = -(v, E' A \tau_{-\delta} \varphi)_\om = -(v, A E'\tau_{-\delta} \varphi)_\om \\
  & = (Av, E'\tau_{-\delta} \varphi)_\om -{ \ip{Dv, E'\tau_{-\delta} \varphi}_W}.
\end{align*}
Note that $E'\tau_{-\delta} \varphi$ satisfies all the
boundary conditions required for it to be in $\V^*$.  
Hence the last
term in the above display is zero. 
We therefore conclude that
$$\ip{ A H_\delta \tau_\delta \tilde E v, \varphi}_{\D(Q_\delta)^{d+1}}
= (\tau_\delta EA v, \varphi)_{Q_\delta},$$ which
proves~\eqref{JGeq:AHtE}.  In particular, 
$H_\delta\tau_\delta \tilde{E} v \in W(Q_\delta)$ whenever $v \in V$.

{\em Step 3. Mollification:} In this step we finish the proof by
considering a $v \in V$ and mollifying the time-translated extension
$\tau_\delta \tilde E v$ constructed above.

To recall the standard symmetric mollifier, let
$\rho_\epsilon \in \D(\mathbb{R}^{d+1})$, for each $\varepsilon >0$ be
defined by
 $$ \rho_\veps(x,t) = \veps^{-(d+1)} \rho_1( \veps^{-1} x,\veps^{-1} t),$$
  where
  \[
  \rho_1(x,t)
  =
  \left\{
    \begin{aligned}
      & k\, \exp\left(-\frac{1}{1-|x|^2-t^2} \right)
      && \text{ if } |x|^2 + t^2<1,
      \\
      & 0
      && \text{ if }|x|^2 + t^2\geq 1,
    \end{aligned}
  \right. 
  \]
  and $k$ is a constant chosen so that $\int_{\RRR^{d+1}} \rho_1 =1.$
  Here $|\cdot|$ is the euclidean norm in $\mathbb{R}^d$.  Let
  $\rho_\veps* v$ denote the function obtained by component-wise
  convolution, i.e, $[\rho_\veps * v]_j = [v]_j * \rho_\varepsilon$
  for all $j$-components. Then $\rho_\veps * v$ is a infinitely smooth
  vector function that satisfies
  \begin{equation}
    \label{JGeq:mollL2}
    \lim_{\veps \to 0} \| v - \rho_\veps * v \|_{\RRR^{d+1} } = 0,
    \qquad 
    \forall \, v \in L^2(\RRR^{d+1})^{d+1}.
  \end{equation}
Consider any 
\[
0 < \delta < \min_{1\leq i \leq d} (a_i/2, T/2),
\]
and define two functions
$ v_\veps = \rho_\veps * \tau_\delta \tilde{E} v$ and
$ a_\veps = \rho_\veps * \tau_\delta \tilde{E} A v.$ Note that the
$ A v_\veps = a_\veps$ on $ \om$ whenever $\veps < \delta/2$, thanks
to \eqref{JGeq:AHtE}.

We now proceed to show that 
\begin{equation}
  \label{JGeq:finalapprox}
  \lim_{\veps \to 0} \left\| v_\veps\big|_\om - v \right\|_W = 0.
\end{equation}
Set
$\delta=3 \veps$ and let
$\veps < \min_{1\leq i \leq d} (a_i/2, T/2)/3$ go to zero. Note that
 \begin{align*}
    \| Av_\veps - Av \|
    & = \| a_\veps - Av \| = 
      \|  \rho_\veps * \tau_\delta \tilde{E} A v - Av \|_\om 
      \\
   & \le 
      \|  \rho_\veps * \tau_\delta \tilde{E} A v - \tilde{E} Av \|_{\RRR^{d+1}}
    \\
    & \le 
      \|  \rho_\veps * \tau_\delta \tilde{E} A v -  
          \tau_\delta \tilde{E} A v \|_{\RRR^{d+1}} 
      + \| \tau_\delta \tilde{E} A v  - \tilde{E} Av \|_{\RRR^{d+1}},
\\
    \| v_\veps - v \|
    &
      \le 
      \| \rho_\veps * \tau_\delta \tilde{E} v  - \tau_\delta \tilde{E} v \|_{\om} 
      + 
      \| \tau_\delta \tilde{E} v - v \|_{\om} 
    \\
    & \le 
      \| \rho_\veps * \tau_\delta \tilde{E} v  - \tau_\delta \tilde{E} v \|_{\RRR^{d+1}} 
      + 
      \| \tau_\delta \tilde{E} v - \tilde{E} v \|_{\RRR^{d+1}}.
  \end{align*}
  Using~\eqref{JGtranslation} and ~\eqref{JGeq:mollL2} it now immediately
  follows that~\eqref{JGeq:finalapprox} holds.

  To conclude, it suffices to prove that $v_\veps|_\om$ is in $\V$.
  Clearly, $\tau_\delta \tilde{E} v$ is identically zero in a
  neighborhood of $\Gamma_0$. Hence we conclude that
  $v_\veps = \rho_\veps * \tau_\delta \tilde{E} v$ vanishes on
  $\Gamma_0$ for small enough $\veps$.  Next, let us examine the value
  of $[v_\veps]_\mu$ at points $(x,t)$ on $\Gamma_b$\blue{, namely}
  \begin{align*}
    [v_\veps]_\mu (x,t) = \int_{\RRR} \int_{\RRR^d} \rho_\veps ( x - x', t-t') [\tau_\delta \tilde E v]_\mu (x', t')\,dx'\,dt'.
\end{align*}
Note that $\rho_\veps(x-x', t-t')$ is a symmetric function of $x'$ about $x$. The other term in the integrand, namely $[\tau_\delta \tilde E v]_\mu (x',\blue{t'})$, is odd about every facet of $\Gamma_b$. Hence the integral of their product vanishes whenever $(x,t) \in \Gamma_b$. Thus, $[v_\veps]_\mu|_{\Gamma_b} = 0 $ and $v_\veps \in \V$. 
\end{proof}

\section{The method and its error estimates}
\label{JGsc:estimates}
In this section, we present the approximation of the previously
described broken weak formulation by the (ideal) DPG method and
provide {\it a priori} and {\it a posteriori} error estimates.

\subsection{The  DPG method}

The ideal DPG method~\cite{JG:Jay2012} seeks $u_h$ and \blue{$\lambda_h$} in
finite-dimensional subspaces $U_h \subset L^2(\om)^{d+1}$ and
$Q_h \subset Q$, respectively, satisfying
\begin{align}
  b( (u_h,\lambda_h), w_h) = F(w_h), \qquad \text{ for all }  w_h \in T(U_h\times Q_h)\blue{,} \label{JG:ideal}
\end{align}
where $T: L^2(\om)^{d+1}\times Q \to W_h$ is such that
$(T(v,\rho),\blue{w } )_{W_h}= b((v,\rho),w)$ for all $ w \in W_h$ and any
$ (v,\rho) \in L^2(\om)^{d+1}\times Q$.  Hereon we denote $U$ to be
$L^2(\om)^{d+1}$ and abbreviate the $W_h$ inner product
$(\cdot,\cdot)_{W_h}$ to simply $(\cdot,\cdot)_h$.

It is well known \cite{JG:DemkoGopal13} that there is a mixed method
that is equivalent to the above Petrov-Galerkin
method~\eqref{JG:ideal}.  One of the variables in this mixed method is
the error representation function $\varepsilon_h \in W_h$ defined by
 \begin{align}
\label{JGeq:eps}
   (\varepsilon_h, w)_h = (f,w) - b((u_h, \lambda_h), w ), \qquad \text{ for all } 
   w \in W_h.
\end{align}
One of the two equations in the mixed formulation given below is a
restatement of this defining equation for $\veps_h$. The mixed
formulation seeks $\varepsilon_h \in W_h$ and
$( u_h, \lambda_h)\in (U_h\times Q_h) $ such that
\begin{align}
\begin{array}{lll}
(\varepsilon_h,w)_h \; +& b((u_h,\lambda_h),w) = F(w) & 
\qquad \text{ for all } w \in W_h,\\
& b((v,\rho),\varepsilon_h)  =  0 & \qquad \text{ for all } (v,\rho) \in U_h\times Q_h.    
\end{array}    \label{JGmixed}
\end{align}
We think of 
\[
\eta = \| \veps_h\|_{W_h} \equiv\left( \sum_{K \in \om_h} \| \veps_h \|_{W(K)}^2\right)^{1/2}
\]
as an {\it a posteriori} error estimator because $\veps_h$ can be
computed from~\eqref{JGeq:eps} after $u_h$ and $\lambda_h$ has been
computed. Alternately, one can view $\veps_h$ as one of the unknowns
together with $u_h$ and $\lambda_h$ as in~\eqref{JGmixed}. 
\blue{Note
that~\eqref{JGeq:eps} implies 
\[
\eta = \sup_{w \in W_h} \frac{ b( (u - u_h, \lambda - \lambda_h), w)}{
  \| w \|_{W_h}},
\]
so it immediately follows that the estimator is globally reliable and 
efficient, namely 
\begin{align*}
\beta  \| (u-u_h, {\lambda}  - \lambda_h)\|_{U\times Q}
  \leq \eta  \leq \|b \|\, \| (u-u_h,{\lambda} - \lambda_h) \|_{U\times Q} 
\end{align*}
where $\beta$ is as in~\eqref{JGeq:beta}. In practice, we use
element-wise norms of $\veps_h$ as {\it a posteriori} element error
indicator.  }


To give an {\it a priori} error estimate with convergence rates, we
need to specify all the approximation subspaces.  We choose the space
$Q_h \subset Q$ by first selecting a finite element space
$V_h \subset V$ and then applying $D_h$ to all functions in it, namely
\[
Q_h = D_h  V_h.
\]
This way we guarantee that $Q_h$ is a subspace of $Q$.  The definition
of $V_h$ and the finite element subspaces of $U$ are based on the type
of elements in $\om_h$. We consider two cases:
\begin{description}
\item [Case~A] $\om_h$ is a geometrically conforming mesh of
  $(d+1)$-simplices:
  \begin{subequations}
    \label{JG:trialsimplex}
    \begin{align}
      V_h &=\{ u \in V\cap C(\bar\om)^{d+1}: u|_{K} \in P_{p+1}(K)^{d+1} \text{ for all } K \in \om_h\}  \\
      U_h &= \{ u \in L^2(\om)^{d+1} : u|_{K} \in P_{p}(K)^{d+1} \text{ for all }  K \in \om_h\}, 
    \end{align}
  \end{subequations}
where $P_p(K)$ is the  space of  polynomials of total degree  $\leq p$ on
$K$.

\item[Case~B] $\om_h$ is a geometrically conforming mesh of
  hyperrectangles.
  \begin{subequations}
    \label{JG:trialrectan}
    \begin{align}
      V_h &=\{ u \in V\cap C(\bar\om)^{d+1}: u|_{K} \in Q_{p+1}(K) ^{d+1}\text{ \blue{ for all }} K \in \om_h \} \\
      U_h &= \{ u \in L^2(\om)^{d+1} : u|_{K} \in Q_{p}(K)^{d+1} \blue{\text{ for all }  K \in \om_h}\},
    \end{align}
  \end{subequations}
  where $Q_p(K)$ is the space of polynomials on $K$ that are of degree
  at most~$p$ in each variable.
\end{description}

Since the wave operator $A$ is a first order differential operator,
$\blue{H^1(\om)^{d+1}} \subset W(\om)$. Hence, the Lagrange finite element space
$V_h$ is contained in $W$. The space $V_h$ has a nodal interpolation
operator $I_h: H^{s+1}(\om)\blue{^{d+1}} \to V_h$ which is bounded for
$s+1 > (d+1)/2$, which we shall use in the proof below.  We will use
 \blue{$C$} to denote a generic mesh-independent constant whose value at
different occurrences may differ. Note that in the estimate of the
theorem below, $h$ is the discretization parameter in both space and
time.

\begin{theorem} \label{JGthm:ee} Suppose
  $u \in V\cap H^{s+1}(\om)^{d+1}$ and $\lambda= D_h u $
  solve~\eqref{JGproblem}. Suppose also that $U_h$ and $V_h$ are set
  as in \eqref{JG:trialsimplex} or \eqref{JG:trialrectan} depending on
  the mesh type, and $Q_h = D_hV_h$. Then\blue{, there exists a constant $C>0$ independent 
  of $h$ such that the discrete solution $u_h\in U_h$ and $\lambda_h \in Q_h$ solving~\eqref{JG:ideal} satisfies} 
  \begin{align}
    \| u - u_h \| + \|\lambda - \lambda_h\|_Q \leq \blue{C} h^s | u|_{H^{s+1} (\om)\blue{^{d+1}}}   
    \label{JG:estimate}
\end{align}
for $(d-1)/2 <s \leq p+1$.
\end{theorem}
\begin{proof}
  The ideal DPG method is quasioptimal, i.e., by \cite[Theorem
  2.2]{JG:Jay2012},
\begin{align*}
\| (u,\lambda) -(u_h, \lambda_h) \|_{U\times Q}^2 &\leq \blue{C} \inf_{(v_h, \rho_h) \in U_h \times Q_h} \|(u,\lambda) - (v_h, \rho_h)\|^2_{U\times Q} \\
&\leq \blue{C}\inf_{(v_h, \rho_h) \in U_h \times Q_h} 
\|u - v_h \|^2 + \|\lambda - \rho_h\|^2_Q.
\end{align*}
The well-known  best approximation estimates for $U_h$ imply
\begin{align}
  \inf_{v_h \in U_h } \| u - v_h\|\leq \blue{C}h^s | u |_{H^s(\om)^{d+1}} , \text{ for all } 0< s \leq p+1.
 \label{JG:erroru}
\end{align} 
To estimate the remaining term, choose $\rho_h = D_h I_h u$. Then, since
$\lambda = D_hu$, by the definition of the $Q$-norm
in~\eqref{JGQnorm} and the Bramble-Hilbert lemma, 
\begin{align}
\nonumber 
\inf_{\rho_h \in Q_h} \|\lambda - \rho_h \|_Q 
&
\le \|u - I_h u \|_W 
  \\ 
  & \label{JG:error2}
    \leq \blue{C} \| u - I_h u \|_{H^1(\om)^{d+1}}  \leq \blue{C}h^s |u|_{H^{s+1}(\om)^{d+1}}
\end{align}
for any $u \in H^{s+1}(\om)$, for $ (d-1)/2 < s \leq p+1$.  Thus, from
\eqref{JG:erroru} and \eqref{JG:error2}, we have that
\eqref{JG:estimate} holds.
\end{proof}

\section{Implementation and numerical results}\label{JGsc:numerical}

We implemented the DPG discretization in the form~\eqref{JGmixed}
with the following change. Since $W_h$ is infinite-dimensional, in
order to get a practical method, we must replace $W_h$ by a
sufficiently rich finite-dimensional space $Y_h^{\blue{m}}$. A full theoretical
analysis of this practical realization of the ideal DPG method is
currently open, but we will provide numerical studies showing its
efficacy in this section. For some \blue{non-negative} integer $\blue{m}$, set $Y_h^{\blue{m}}$ as
follows.
\begin{itemize}
\item In Case~A (see~\eqref{JG:trialsimplex}) we set 
  $Y_h^{\blue{m}}= \{ w \in W_h(\om) : w|_{K} \in P_{\blue{m}}(K)\blue{^{d+1}}\},$
\item In Case~B (see~\eqref{JG:trialrectan}) we set
  $Y_h^{\blue{m}} = \{ w \in W_h(\om) : w|_{K} \in Q_{\blue{m}}(K)\blue{^{d+1}}\}$.
\end{itemize}
Then, we compute $e_h \in Y_h^{\blue{m}}$, $u_h \in U_h$ and $\lambda_h 
\in Q_h$ satisfying 
\begin{align}
\begin{array}{lll}
(e_h,w)_h \; +& b((u_h,\lambda_h),w) = F(w) & \qquad \text{ for all } w \in
Y_h^{\blue{m}},\\
& b((v,\rho),e_h)  =  0 & \qquad \text{ for all } (v,\rho) \in U_h\times Q_h.    
\end{array}    \label{JGmixed-h}
\end{align}
In our numerical experience, the choice ${\blue{m}} = p+d+1$ gave optimal
convergence rates (as reported in detail below). This choice is
motivated by the study in~\cite{JG:Jay14}.  The  choice ${\blue{m}}=p+d$ did not
give optimal convergence rates for $p>2$ and $d=1$. A brief report of
the performance of an adaptive algorithm is also included in the $d=1$
case. Here again, we observed marked deterioration of adaptivity if
${\blue{m}}=p+d$ is used instead of ${\blue{m}}=p+d+1$ for higher degrees.
 Beyond these comments, we shall not describe these negative results further, but
will henceforth focus solely on the ${\blue{m}}=p+d+1$ case. All the numerical
results have been implemented using the NGSolve \cite{JG:NGSolve}
finite element software and the codes used for the experiments below
are available in \cite{JG:Gopalother15}.

\subsection{A null space} 

In order to implement~\eqref{JGmixed-h}, one strategy is to set
$\lambda_h = D_h z_h$ for some $z_h \in V_h$ and solve 
\begin{align}
\begin{array}{lll}
(e_h,w)_h \; +& b((u_h,D_h z_h),w) = F(w) & \qquad \text{ for all } w \in
Y_h^{\blue{m}},\\
& b((v,D_h r ),e_h)  =  0 & \qquad \text{ for all } v \in U_h, r \in V_h.
\end{array}    \label{JGmixed-Dh}
\end{align}
We can decompose $V_h$ into interior ``bubbles'' in
$V_h^0 = \{ z \in V_h: z|_{\partial K}=0 $ for all $K \in \om_h \}$,
and a remainder $V_h^1 \equiv V_h / V_h^0$. Since
$b( (v, D_h V_h^0), w) = 0$, we may replace $V_h$ by $V_h^1$
in~\eqref{JGmixed-Dh} (and compute a $z_h \in V_h^1$).  Let
$\{ y_k\}, \{u_i\},$ and $\{z_j\}$ denote a local finite element basis
for $Y_h^{\blue{m}}$, $U_h$ and $V_h^1$, respectively. Using this basis, the
system~\eqref{JGmixed-Dh} with $V_h$ replaced by $V_h^1$, yields a
matrix equation of the following form
\begin{align}
\label{JGeq:mat}
\bb{ \texttt{A} & \texttt{B}\\ 
\blue{\tt{B^T}} &\texttt{0} } \bb{ \texttt{e} \\ 
\texttt{x} } = \bb{\texttt{f} \\ \texttt{0} },
\end{align}
where $\tt e$ and $\tt x$ are the vectors of coefficients in the basis
expansion of $e_h \in Y_h^{\blue{m}}$ and $(u_h, z_h) \in U_h \times V_h$,
respectively, $\mathtt{A}_{kl} = (y_l, y_k)_h$,
$[\mathtt{B_0}]_{ki} = b((u_i, 0), y_k)$,
$[\mathtt{B_1}]_{kj} = b((0, D_hz_j), y_k)$, and
$\mathtt{B} = [\mathtt{B_0}, \mathtt{B_1}]$. In all our numerical
experiments, for the above-mentioned choice of ${\blue{m}}=p+d+1$, we observed
that the matrices $\mathtt{A}$ and $\mathtt{B_0}$ have trivial null
spaces.

However, we caution that ${\tt{B_1}}$ may have a null space. This runs
contrary to our experience with DPG methods on non-spacetime problems,
so we expand on it.  Note that (cf.~\eqref{JGbdryop})
\[
[\mathtt{B_1}]_{kj} = b((0, D_hz_j), y_k)
= \sum_{K \in \om_h} \int_{\partial K} {\tt{D_{x,t}}} z_j \cdot y_k
\]
where 
\begin{align*}
\tt{ D_{x,t}= \bb{n_t \blue{\tt{I}_d}& -\blue{c}n_x\\ -\blue{c}n_x^T & n_t  }}
\end{align*}
\blue{and $ \tt{I}_d$ is the $ d \times d $ identity matrix.}
It is immediate that on mesh facets with certain combinations of $n_x$
and $n_t$, the matrix ${\tt{D_{x,t}}}$ is singular.  Then ${\tt{B_1}}$
will  have a nontrivial kernel.

As an example, in Figure~\ref{JG:shapekernel}, we show one of the
$z_j$ that is in the null space of $\tt{B_1}$ on a triangular mesh for
$p=1$ \blue{and $c=1$}. In fact, on the mesh shown, there are 8 basis functions of
$V_h^1$ that are in the null space of $\tt{B_1}$, two for each
diagonal edge.  Recall that the wave speed is $1$ for our model wave
problem, so these edges align with the light cone for $d=1$. In the
case of $d=2$ space dimensions, we continued to find a nontrivial null
space for $\tt{B_1}$ on analogous meshes.
 
\begin{figure}[ht]
\centering
\includegraphics[width=4cm,trim=5.3cm 2cm 5.3cm 2cm, clip =true]{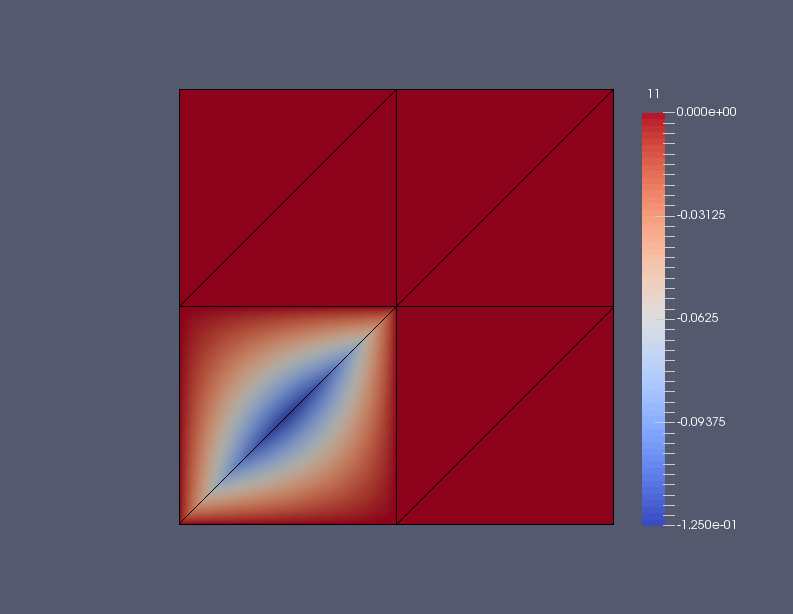}
\caption{Example of a spacetime shape function $z_j$ in  the kernel} \label{JG:shapekernel}
\end{figure}

This null space problem occurs because the interface variable
$\lambda_h$ is set indirectly by applying the singular operator $D_h$
on $V_h$. If one could directly construct a basis for $Q_h = D_h V_h$,
then one can directly implement~\eqref{JGmixed-h} (instead
of~\eqref{JGmixed-Dh}). However, we do not know how to construct such
a basis easily on general simplicial meshes. Hence we continue on to
describe how to solve~\eqref{JGmixed-Dh} despite its kernel.

\subsection{Techniques to solve despite the null space}

Despite the above-mentioned problem, one may solve the DPG system
using one of the following approaches.

\subsubsection{Technique~1: Remaining orthogonal to null space in
  conjugate gradients}

The matrix system~\eqref{JGeq:mat} can be solved by reducing it to its
Schur complement
\begin{align} \label{JGeq:Schur}
\tt
\blue{B^T} A^{-1} B x = B^T A^{-1} f
 \end{align}
first. Let $\tt{C} =\tt{\blue{B^T} A^{-1} B}$ and ${\tt{g = B^T A^{-1} f}}$.
The matrix ${\tt{C}}$ is symmetric and positive semi-definite. Its easy to see that 
\[
\ker {\tt{C}} = \ker{\tt{B}}.
\]
Thus solutions of~\eqref{JGeq:Schur} are defined only up to this
kernel. Note however that since $\ker{\tt{B}} = \ker\tt{B_1}$ and
${\tt{B_0}}$ has only the trivial kernel, the $U_h$-component of the DPG
solution is uniquely defined independently of $ \ker\tt{B_1}$.

One may obtain one solution of~\eqref{JGeq:Schur} using the conjugate
gradient method, which computes its $n$th iterate ${\tt{x}}_n$ in the
Krylov space
\[ 
K_n({\tt C, r_0}) = \mathrm{span}\{ {\tt C^k r^0} : \, k = 0,\dots, n-1 \}
\]
where $\tt {r_0= g - A x_0} $ is the initial residual. This iteration
will converge if $K_n({\tt C, r_0})$ remains ($\ell^2$) orthogonal to
$\ker({\tt{C}})$ for all $n$. A simple prescription to guarantee this
orthogonality is to choose the initial iterate $\tt{x_0=0}.$ Indeed,
if $\tt x_0 =0$, then $\tt {r_0} =$ $\tt g$ $= \tt{B^T A^{-1} f}$ is in the range
of $\tt{B^T}$ which equals the orthogonal complement of
$\ker \tt B = \ker \tt C$. Then for all $n \ge 1$, its obvious that
${\tt{C}}^n {\tt{r_0}}$ is also orthogonal to $\ker \tt C$. Thus
$K_n({\tt C, r_0})$ is orthogonal to $\ker \tt C$.

To summarize this technique, we use the conjugate gradient algorithm
to compute one solution orthogonal to $\ker (\tt C)$ and extract the
unique $U_h$-component from that solution for reporting the errors.

\subsubsection{Technique~2: Regularization of the linear system}

Another technique to solve the singular system \eqref{JGeq:Schur}
approximately is regularization. First, we rewrite \eqref{JGeq:Schur} in block form 
as
\[
  \begin{bmatrix}
  \tt{B_0^T A^{-1} B_0} & \tt{B_0^T A^{-1} B_1 }\\
  \tt{B_1^T A^{-1} B_0} & \tt{B_1^T A^{-1} B_1}
  \end{bmatrix}
  \tt{x = g}.
\]
Since only $\tt{B_1}$ may have a nontrivial kernel in $V_h^1$, we can
convert this to an invertible system by adding a small
positive-definite term in $V_h^1$. Namely, let $\tt M$ be the mass
matrix ${\tt{M}}_{jl} = (z_l, z_j)$.  Instead of solving
\eqref{JGeq:Schur}, we solve for
\begin{equation}
  \label{JGeq:reg}
  \begin{bmatrix}
  \tt{B_0^T A^{-1} B_0} & \tt{B_0^T A^{-1} B_1 }\\
  \tt{B_1^T A^{-1} B_0} & \tt{B_1^T A^{-1} B_1 + \alpha M}
  \end{bmatrix}
  \tt{x = g}.
\end{equation}
where $\tt \alpha$ is a positive regularization parameter, usually set
much smaller than the order of the expected discretization errors. In
all our reported experiments  it was set to $10^{-9}$.
The regularized system~\eqref{JGeq:reg} is invertible and can be
solved using any direct or iterative methods.


\subsection{Convergence rates in two-dimensional spacetime}

Let $\om =(0,1)^2$. We consider a problem with homogeneous boundary
and initial conditions where the exact solution to the second order
wave equation is given by $\phi(x,t) = \sin(\pi x)\sin^2(\pi t)$.
Then, the exact solution for the first order system is 
\[
u = 
\begin{bmatrix}
 \blue{c} \pi \cos(\pi x) \sin^2(\pi t)
  \\
\pi \sin(\pi x) \sin(2\pi t)
\end{bmatrix}
\]
and the corresponding source terms are
\[
g = 0, 
\qquad 
f = \pi^2\sin(\pi x)(2\cos(2\pi t)+\blue{c^2}\sin^2(\pi t)).
\]
In each experiment, a (non-uniform) coarse triangular mesh of $\om$
was constructed, with element diameters not exceeding a reported mesh
size~$h$ \blue{ and consider $c=1$}.  Successive refinements of the mesh were obtained by
connecting the mid points of the edges.  

We observe in Table~\ref{JGtable:triang} that the order of convergence
for $u_h$ in the $L^2$ norm is $O(h^{p+1})$ in accordance with
Theorem~\ref{JGthm:ee}.  Similarly in Table~\ref{JGtable:rectang}, we
observe the same convergence rates for rectangular meshes.  All
results in both tables were obtained using Technique~1.


\begin{table}[ht]
\centering
\begin{footnotesize}
\begin{tabular}{|c|c|c||c|c||c|c||c|c|} \hline
$h$   &  $p = 0 $&  Order &  $p=1$ & Order & $ p =2$ & Order  & $p = 3$ &Order \\ \hline 
1/4   &  1.2849e+00     &-- &   1.5371e-01  &    --   & 2.0385e-02  &    --    & 1.2619e-03 & -- \\ 
 1/8  &  5.6379e-01  & 1.19& 5.6127e-02 &  1.45  & 4.7540e-03   & 2.10 &1.5370e-04 &+3.04    \\
 1/16 & 2.2067e-01   & 1.35 &1.2472e-02 &  2.17  &  5.4897e-04  & 3.11 & 7.8519e-06 & +4.29  \\  
1/32  &  1.0214e-01    & 1.11 & 3.0308e-03&  2.04 &  6.6955e-05   &  3.00  & 4.7863e-07 & +4.04   \\
\hline
\end{tabular}  
\end{footnotesize}
\caption{Convergence rates for $\|u - u_h\|$ on triangular meshes 
  using Technique~1.\label{JGtable:triang}}
\end{table}
\begin{table}
\centering
\begin{footnotesize}
\begin{tabular}{|c|c|c||c|c||c|c||c|c|} \hline
$h$ &  $p = 0 $&  Order &  $p=1$ & Order& $ p =2$ & Order & $p = 3$ &Order  \\ \hline 
1/4  & 9.7226e-01 &--   &1.6834e-01& -- &6.6722e-03& -- &  2.0910e-03 & --\\
1/8  & 4.7357e-01 &1.04 &4.2869e-02&1.97&8.5059e-04&2.97 & 1.3308e-04 & 3.97 \\
1/16 & 2.3291e-01 &1.35 &1.0763e-02&1.99&1.0707e-04&2.99 &8.3773e-06 & 3.99 \\
1/32 & 1.1587e-01 &1.11 &2.6935e-03&2.00&1.3409e-05&3.00 &5.2613e-07  &3.99 \\ 
\hline 
\end{tabular}  
\end{footnotesize}
\caption{Convergence rates for $\| u - u_h \|$ on rectangular
  meshes using Technique~1\label{JGtable:rectang}.}
\end{table}

\subsection{Adaptivity}

\begin{figure}[h]
\begin{minipage}[c]{0.7\textwidth}
\begin{subfigure}{.45\linewidth}
\centering
\includegraphics[width=4cm,trim=5cm 2cm 5.3cm 2cm, clip = true]{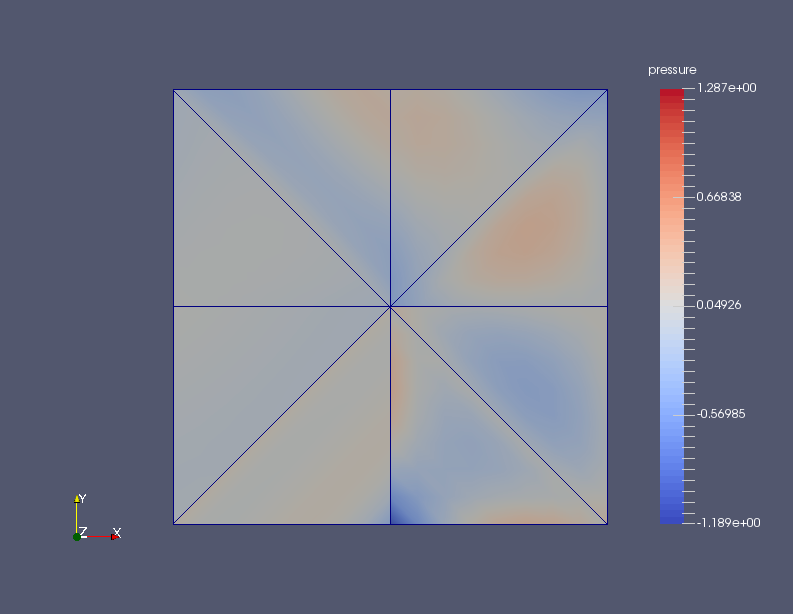}
\caption{0 refinements\quad} 
\end{subfigure}
\begin{subfigure}{.45\linewidth}
\centering
\includegraphics[width=4cm,trim=5cm 2cm 5.3cm 2cm, clip =true ]{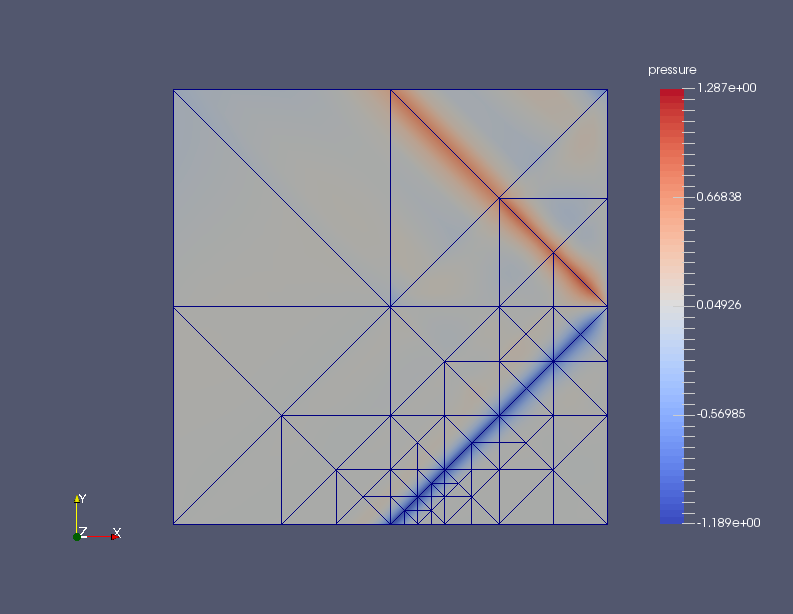}
\caption{6 refinements} 
\end{subfigure}\\
\begin{subfigure}{.45\linewidth}
\centering
\includegraphics[width=4cm,trim=5cm 2cm 5.3cm 2cm, clip =true]{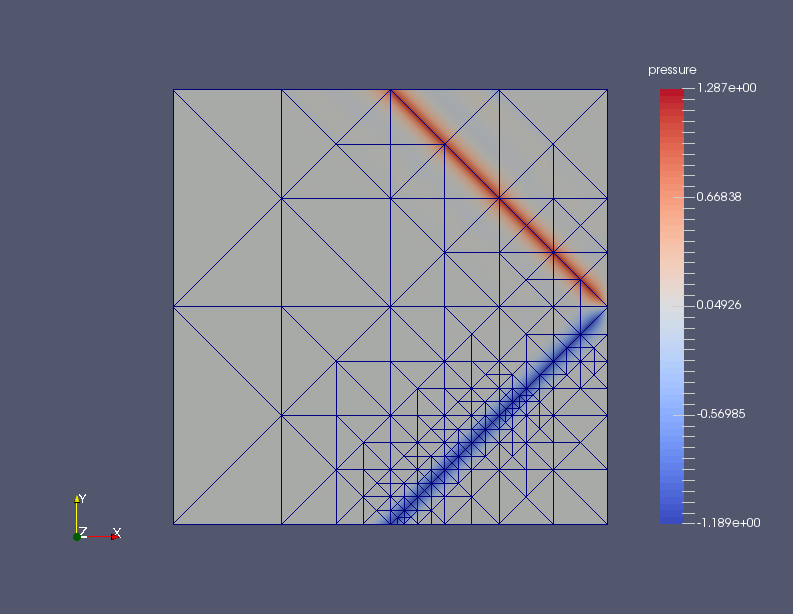}
\caption{14 refinements \qquad} 
\end{subfigure}
\begin{subfigure}{.45\linewidth}
\centering
\includegraphics[width=4cm,trim=5.3cm 2cm 5.3cm 2cm, clip =true]{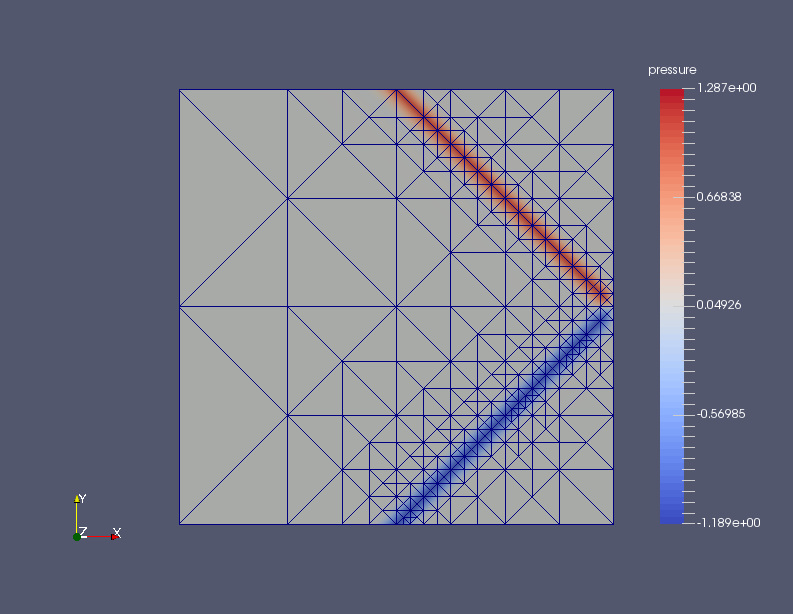}
\caption{ 22  refinements} 
\end{subfigure}
\end{minipage}
\begin{minipage}[c]{0.15\textwidth}
\begin{subfigure}{1\linewidth}
\vspace{-.55cm}
\includegraphics[width=3cm,trim=22cm 2cm 0cm 2cm, clip = true]{figs/pressure0}
\end{subfigure}\\
\end{minipage}
\caption{\blue{Iterates from the adaptive algorithm. Numerical
    pressure $\mu$ is shown for $p =3$. Time axis is vertical.}} \label{JG:adapt}
\end{figure}

Let $\om = (0,1)^2$. We consider the same model
problem~\eqref{JG:model} but now with zero sources $f=g=0$
and the non-zero initial condition 
\[
  \mu|_{t=0} =-\phi_0, 
\qquad 
q|_{t=0} = \phi_0
\]
in place of \eqref{JGmodel-bc}, where
$\phi_0 = \exp(-1000 ((x - 0.5)^2))$.  The boundary condition $\mu = 0$
continues to remain the same. This simulates a beam reflecting off the
Dirichlet boundary.

In Figure~\ref{JG:adapt}, we display a few iterates from the standard
adaptive refinement algorithm using $p =3$ and the DPG error
estimator.  We started with the extremely coarse mesh shown in
Figure~\ref{JG:adapt}(a), used the element-wise norms of $e_h$ to
compute the DPG element error indicator, marked elements with more
than 50\% of the total indicated error, refined the marked elements
(and more for conformity) by bisection, and repeated this adaptivity
loop.  The few iterates from the adaptivity loop shown in
Figure~\ref{JG:adapt} show the potential of the spacetime DPG method
to easily capture localized features in spacetime.
\subsection{ Adaptivity with inhomogeneous materials}
Consider  the case when the domain consists of two regions, namely $\om_l = (0,0.5) \times(0,1.4)$, and $\om_r = (0.5,1)\times(0,1.4)$, and a more general first order wave equation   

\[ 
\begin{bmatrix}
  	\kappa_1 & 0 \\
  	 0 & \kappa_2 
  	  \end{bmatrix} 
  \dt u -  \begin{bmatrix}
           0 &c \\
           c & 0 
           \end{bmatrix}
 \dx u = 0,
 \] 
 where 
 \[
  \kappa_1 = \left\{ \begin{matrix} 2, & 0<x<1/2 \\ 1/2, & 1/2 <x<1, \end{matrix} \right. 
    \qquad
  \kappa_2 = \left\{ \begin{matrix} 2, & 0<x<1/2 \\ 1/2, & 1/2 <x<1, \end{matrix} \right. 
\]
as in \cite{JG:Jay15}, we set $c =1$. Here, $\kappa_1, \kappa_2$ are material parameters. The wave speed is given by $c/\sqrt{\kappa_1 \kappa_2}$, and jumps between 0.5 to 2. The impedance,  given by $\kappa_1 / \kappa_2$, is the same in both regions, therefore we  expect no reflections between the regions. We set vanishing Dirichlet boundary conditions as the previous example  and 
$$ f = g =  0, \quad u_q(x,0) = e^{-5000((x-0.2)^2)}, \text{ and } u_\mu(x,0) = - e^{-5000((x-0.2)^2)}. $$
We can observe the results of adaptive algorithm in Figure \ref{JG:figs:Fig4}.



\begin{figure}
\begin{minipage}[c]{0.8\textwidth}
\begin{subfigure}{.25\linewidth}
\includegraphics[width=3.5cm,trim=0cm 0cm 4cm 1cm, clip = true]{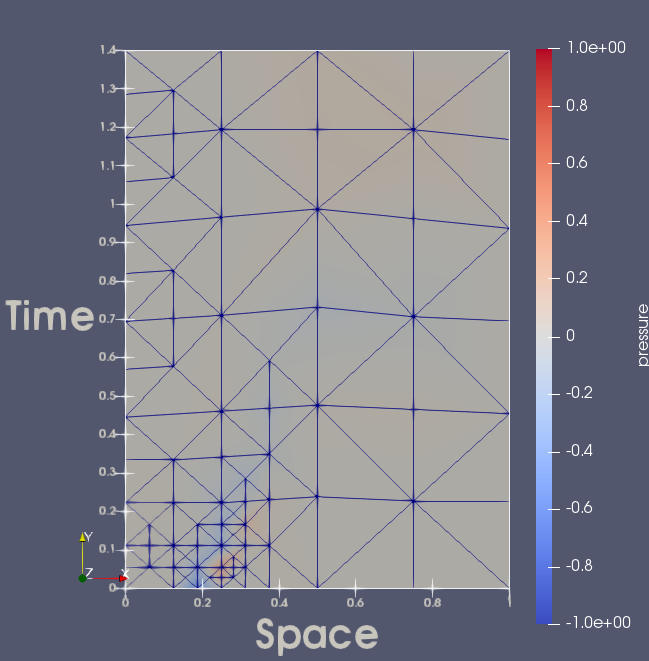}
\caption{5 refinements}
\end{subfigure} \quad
\begin{subfigure}{.25\linewidth}
\includegraphics[width=3.5cm,trim=0cm 0cm 4cm 1cm, clip = true]{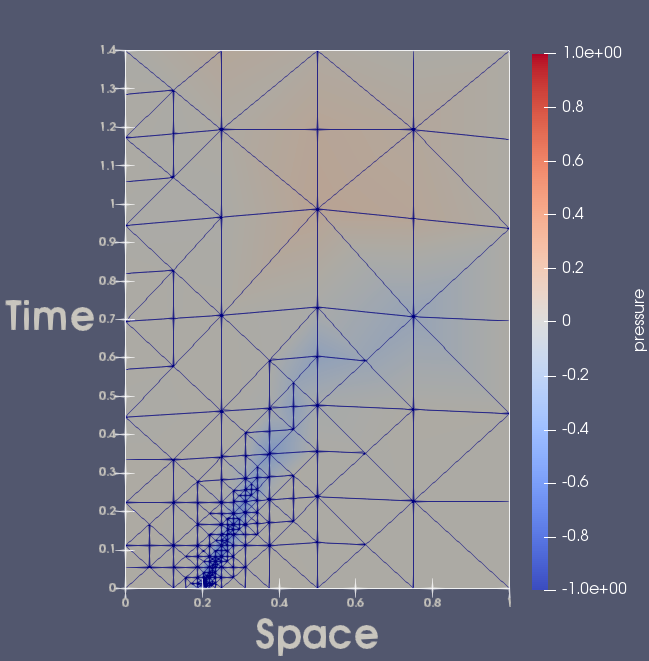}
\caption{10 refinements}
\end{subfigure} \quad
\begin{subfigure}{.25\linewidth}
\includegraphics[width=3.5cm,trim=0cm 0cm 4cm 1cm, clip = true]{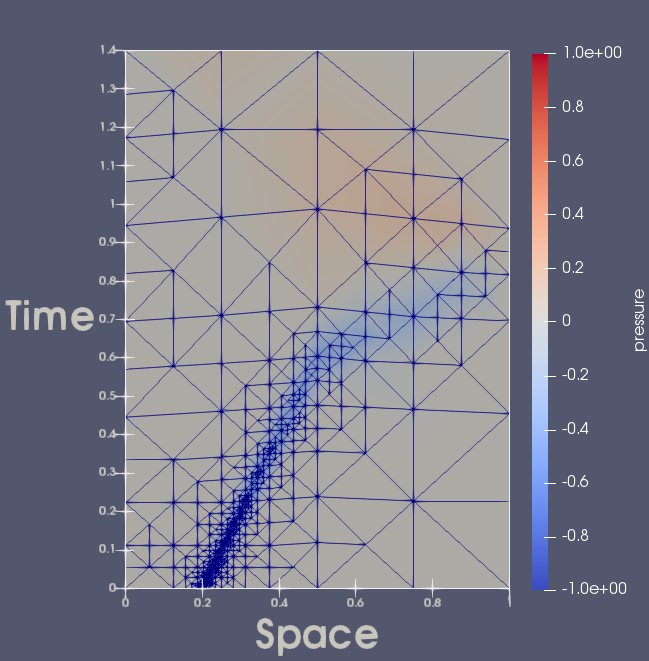}
\caption{15 refinements}
\end{subfigure} \\
\begin{subfigure}{.25\linewidth}
\includegraphics[width=3.5cm,trim=0cm 0cm 4cm 1cm, clip = true]{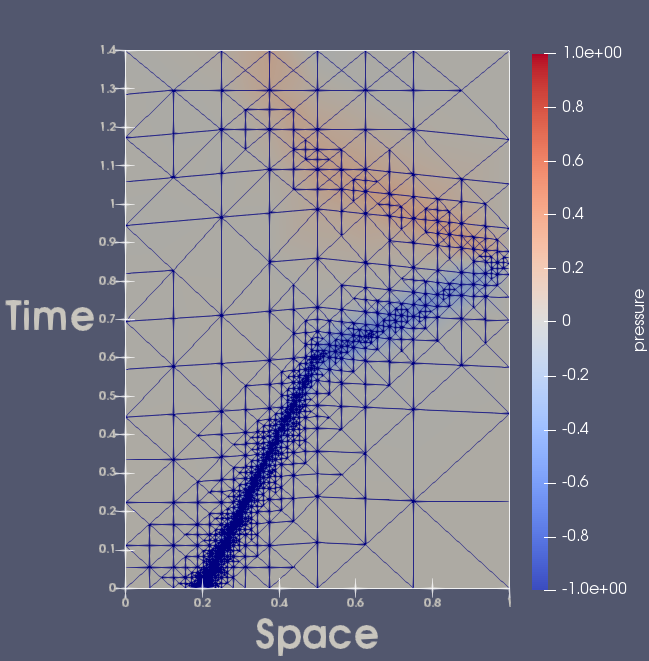}
\caption{20 refinements} 
\end{subfigure} \quad 
\begin{subfigure}{.25\linewidth}
\includegraphics[width=3.5cm,trim=0cm 0cm 4cm 1cm, clip = true]{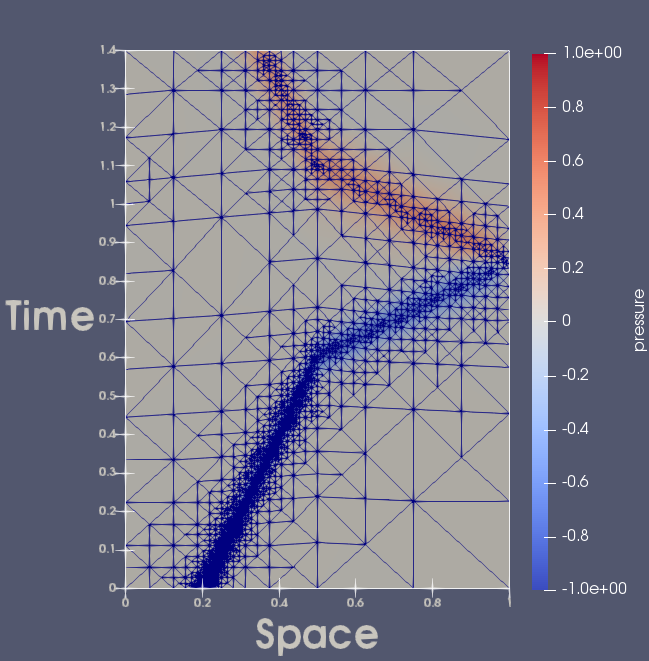}
\caption{25 refinements}
\end{subfigure} \quad
\begin{subfigure}{.25\linewidth}
\includegraphics[width=3.5cm,trim=0cm 0cm 4cm 1cm, clip = true]{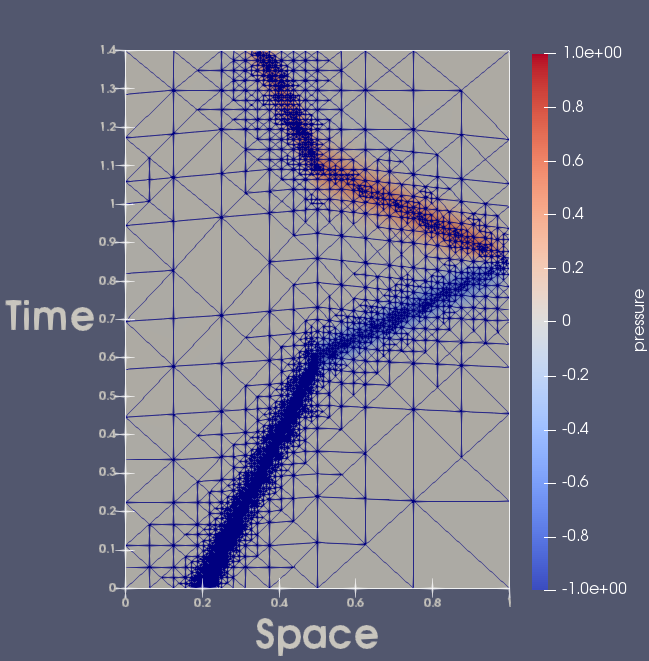}
\caption{30 refinements}
\end{subfigure} 
\end{minipage}
\begin{minipage}[tl]{0.15\textwidth}
\vspace{-1cm}
\includegraphics[width=1.8cm,trim=18cm 0cm 0cm 0cm, clip = true]{figs/adapt30}
\end{minipage}
\caption{Iterate from the adaptive algorithm. Numerical pressure $\mu$ is shown for $p =1$.} \label{JG:figs:Fig4}
\end{figure}

\subsection{Convergence rates in three-dimensional spacetime}

On $\om = (0,1)^3$, we consider the problem where the exact solution
to the second order wave equation is given by
$\phi(x,t) = \sin(\pi x)\sin(\pi y) t^2.$ This corresponds to 
\[
u = \bb{\pi \cos(\pi x) \sin(\pi y) t^2 \\ \pi \cos(\pi y) \sin(\pi  x) t^2
\\ 
2\sin(\pi x) \sin (\pi y) t},
\]
$f = \sin(\pi x) \sin(\pi y )( 2 + 2 \pi^2t^2) $ and $g=0$.

In Table~\ref{JGtable:tet}, we show the convergence rates of $u_h$ for
successively refined tetrahedral meshes, obtained using Technique~2
for $p=0,1,2,3$. 
Table~\ref{JGtable:hexa} shows analogous results obtained for
successively refined hexahedral meshes using Technique~1.  In all
these cases, we observe $O(h^{p+1})$ convergence rates for~$u_h$.

\begin{table}[hb]
\centering
\begin{footnotesize}
\begin{tabular}{|c|c|c||c|c||c|c||c|c|} \hline
$h$ &  $p=0$  &  Order & $p=1$ &  Order & $p=2$ & Order & $p=3 $&  Order\\ \hline
1   & 9.0604e-01&  -- &4.7829e-01 & -- & 1.4146e-01&  --  & 4.3952e-02 & --   \\
1/2 & 6.0557e-01&0.58 &1.3924e-01 &1.78& 1.3912e-02& 3.35 & 3.2845e-03 & 3.74 \\
1/4 & 3.3896e-01&0.84 &3.3508e-02 &2.05& 1.4769e-03& 3.24 & 1.6490e-04 & 4.32\\
1/8 & 1.5469e-01&1.13 &8.9554e-03 &1.90& 1.7210e-04& 3.10 & 9.9691e-06 & 4.05 \\ \hline
 \end{tabular}  
\caption{Convergence rates for $\|u-u_h\|$ on  tetrahedral meshes
  obtained   using Technique~2.} \label{JGtable:tet}
\end{footnotesize}
\end{table}
\begin{table}[hb]
\centering 
\begin{footnotesize}
\begin{tabular}{|c|c|c||c|c||c|c||c|c|} \hline
$h$ &$p = 0 $  &  Order &   $p=1$& Order& $ p =2$   & Order& $p=3$ & Order \\ \hline 
1   &1.1149e+00&  -- &6.0068e-01 & --   &2.8828e-02 &  --  &3.3262e-02 &   -    \\
1/2 &7.5769e-01&0.56 &1.5124e-01 &1.99  &2.8264e-03 & 3.35 &2.0540e-03 & 4.02  \\
1/4 &4.2035e-01&0.85 &3.8592e-02 &1.97  &3.5256e-04 & 3.00 &1.3234e-04 & 3.96 \\
1/8 &2.1338e-01&0.98 &9.6918e-03 &1.99  &3.8023e-05 & 3.21 &9.3766e-06 & 3.82 \\ \hline
 \end{tabular}  
\end{footnotesize}
\caption{ Convergence rates for  $\|u - u_h\|$ on  hexahedral meshes
  using Technique~1.\label{JGtable:hexa}}
\end{table}

\subsection{Adaptivity in 3D} Consider $\om = (0,1)^3$, and  the problem where the
 exact solution is given by $$u(x,y,t)= e^{-200((x-x_0-ct)^2+(y-y_0-ct)^2)}\bb{1 \\1\\ -1}$$
 Here we have chosen $x_0 = y_0 = 0.2$. This corresponds to set $f = 0$, and 
 $$g = 400c\,e^{-200((x-x_0-ct)^2+(y-y_0-ct)^2)}\bb{y -y_0 - ct \\x -x_0 - ct } .$$
 After setting $c = 1/2$ and homogeneous Dirichlet boundary conditions, we  observe that the adaptive scheme  captures with precision the behavior of the wave propagation in Figure~\ref{JG:fig:Fig5} .
 
 \blue{
\begin{figure}[h]
\includegraphics[width=4.5cm,trim=1cm 0cm 6cm 1.2cm, clip = true]{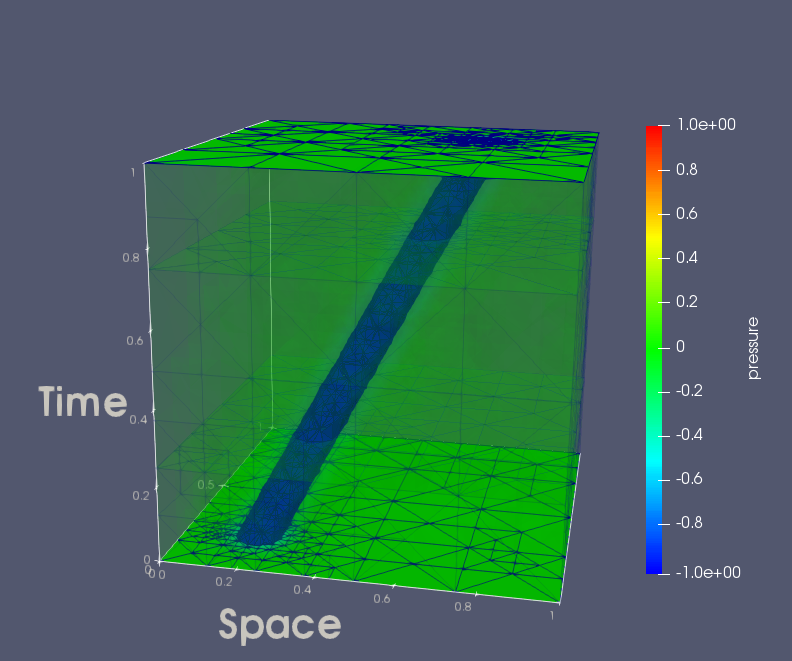}
\includegraphics[width=4.5cm,trim=1cm 0cm 6cm 1cm, clip = true]{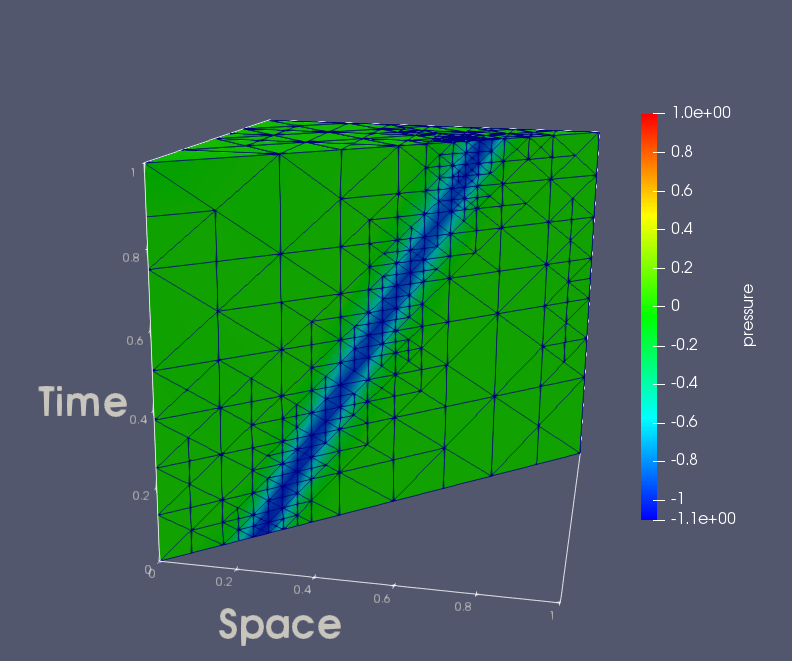}
\includegraphics[width=1.55cm,trim=22cm 2.5cm 0.3cm 3.5cm, clip = true]{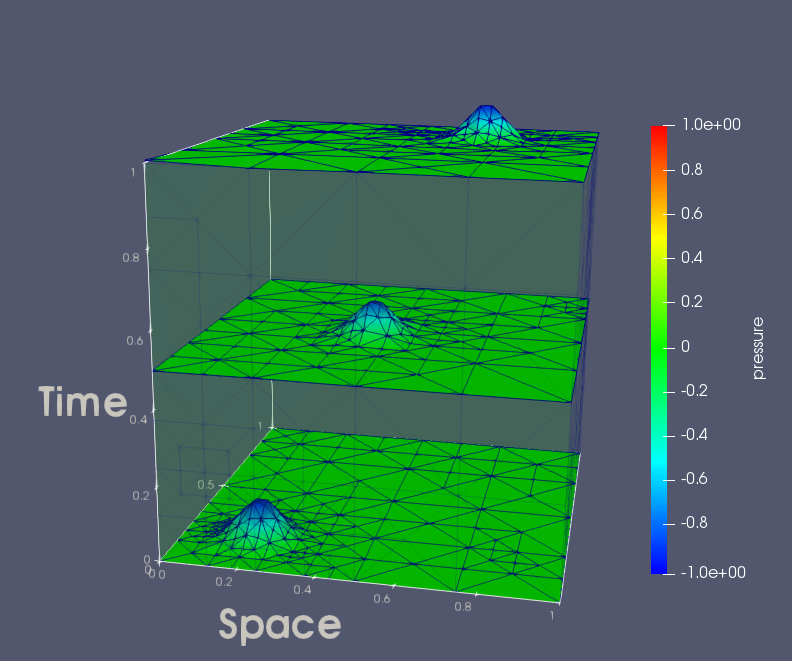}
\caption{Adaptivity example in three dimensions after 10 iterations, $u_\mu$ component is shown, and $p = 1$.  
} \label{JG:fig:Fig5}
\end{figure}
}
\subsection{Adaptivity with variant wave speed} 
 Consider $\om = (-4,4)^2\times(0,8)$ and
 the exact solution $$u(x,y,t)  = \bb{-1\\-1\\1}e^{-20((x-c\cos(\frac{\pi}{2}t))^2 +(y+c\sin(\frac{\pi}{2}t))^2)}.$$
  After setting  $c = 1$ , the component $u_\mu$  corresponds to a pulse propagating from the coordinates $(1,0,0)$ (at time $t = 0$), to $(1,0,8)$ (at time $t = 8$), rotating in time with respect to the $t$-axis with distance equals 1.  We have chosen the solution so we  observe two complete rotations from $t = 0$ to $t = 8$.  (see Figure \ref{JG:figs:Fig6}). 

Consider  the initial condition  $u (x,y,0) = \bb{-1\\-1\\1} e^{-20((x-1)^2 +y^2)}$, 
with homogeneous Dirichlet boundary conditions,
and set
\begin{align*}
g &= 40c  \bb{x (\frac{\pi}{2}\sin(\frac{\pi}{2}t)+1) +\cos(\frac{\pi}{2}t)( \frac{\pi}{2}y-c)) \\
\sin( \frac{\pi}{2}t)( \frac{\pi}{2}x+c) +y (1+ \frac{\pi}{2}\cos( \frac{\pi}{2}t)) } u_\mu (x,y,t), \\
 f&=  -40c  \left( x \left(\frac{\pi}{2}\sin\left(\frac{\pi}{2}t\right)+1\right)
                             +y \left( \frac{\pi}{2}\cos\left( \frac{\pi}{2} t \right) +1\right)
                             -\sqrt{2}c\sin\left(\frac{\pi}{2}+\frac{1}{4} \right) \right)u_\mu (x,y,t).
\end{align*}

\begin{figure}[h]
\begin{subfigure}[c]{.4\linewidth}
\includegraphics[width=5.5cm,trim=0cm 0cm 3.3cm 1cm, clip = true]{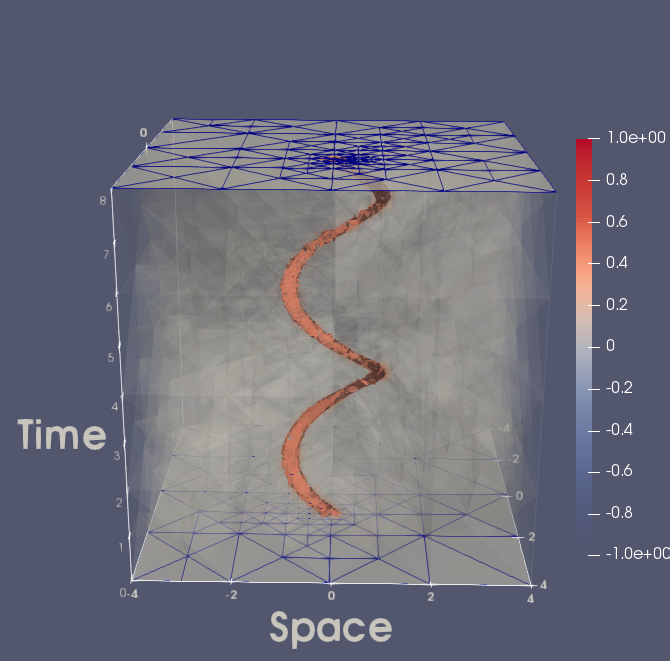}
\caption{Spacetime contour of the solution $u_\mu$.}
\end{subfigure} \quad
\begin{subfigure}{.4\linewidth}
\includegraphics[width=5.5cm,trim=0cm 0cm 3.3cm 1cm, clip = true]{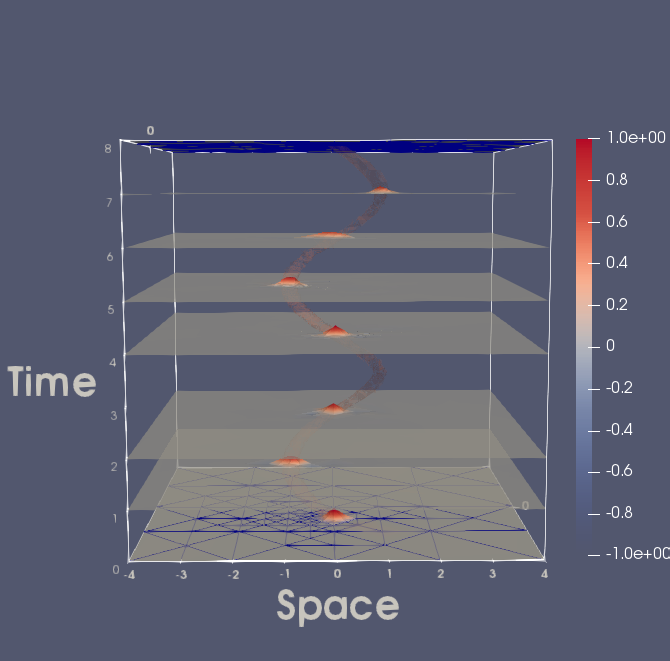}
\caption{Slides of the solution $u_\mu$ at particular time steps.}
\end{subfigure} 
\quad 
\begin{subfigure}[c]{.1\linewidth}
\vspace{-1.05cm}
\includegraphics[width=1.12cm, trim=18cm 2cm 1cm 0cm, clip = true]{figs/adapt30}
\end{subfigure}
\caption{Iterate from the adaptive algorithm. Numerical pressure $\mu$ is shown for $p =1$.} \label{JG:figs:Fig6}
\end{figure}

\end{document}